``` 

```
\documentclass[12pt]{amsart} 
\usepackage{geometry}
\usepackage{amsmath,amsthm,amssymb,mathrsfs,ifthen,verbatim,enumerate}
\usepackage{tikz,graphicx} 
\usepackage{hyperref}
\usetikzlibrary{calc,positioning,arrows}
\usepackage{multicol}
\usepackage{pdfpages}
\usepackage{placeins} 

\allowdisplaybreaks 

\usepackage[labelfont=sc, justification=centering, textformat=period]{caption}
\makeatletter \g@addto@macro\@floatboxreset\centering \makeatother 

\usepackage{cleveref} 
\crefname{subsection}{section}{sections}
\Crefname{subsection}{Section}{Sections}
\renewcommand\eqref[1]{(\ref{#1})}

\newtheorem{thm}{Theorem}[section]
\newtheorem{cor}[thm]{Corollary}
\newtheorem{lem}[thm]{Lemma}
\newtheorem{prop}[thm]{Proposition}

\theoremstyle{definition}
\newtheorem{defn}[thm]{Definition}
\newtheorem{remark}[thm]{Remark}

\newcommand\<{\begin{equation}} \renewcommand\>{\end{equation}}

\newcommand\abs[1]{\left\lvert #1 \right\rvert}
\newcommand\floor[1]{\left\lfloor #1 \right\rfloor}
\newcommand\ceil[1]{\left\lceil #1 \right\rceil}
\newcommand\bra[1]{\left\lfloor #1 \right\rceil}
\newcommand\ang[1]{\left\langle #1 \right\rangle}
\newcommand\setform[3][:]{\left\{\,#2 \,#1\, #3 \,\right\}}

\renewcommand\bar[1]{{\,\overline{\!#1\!}\,}}

\let\hat\widehat

\renewcommand\Re{\operatorname{Re}}
\renewcommand\Im{\operatorname{Im}}

\renewcommand\i{{\begin{tikzpicture}[line width=0.04em, line join=round, line cap=round] \draw [opacity=0] (-0.1,0) -- (0.1,0); \draw (-0.025,0) -- (-0.025,0.1765); \draw (-0.05,0) -- (0.05,0); \draw (-0.05,0.17) -- (0.025,0.19) -- (0.025,0); \draw [rotate=-30] (120:0.27) circle (0.025 and 0.03); \end{tikzpicture}}} 
\newcommand\ssi{{\begin{tikzpicture}[scale=0.67, line width=0.028em, line join=round, line cap=round] \draw [opacity=0] (-0.1,0) -- (0.1,0); \draw (-0.025,0) -- (-0.025,0.1765); \draw (-0.05,0) -- (0.05,0); \draw (-0.05,0.17) -- (0.025,0.19) -- (0.025,0); \draw [rotate=-30] (120:0.27) circle (0.025 and 0.03); \end{tikzpicture}}}

\providecommand\C{C} 
\renewcommand\C{{\mathbb C}}
\newcommand\D{{\mathbb D}}
\newcommand\N{{\mathbb N}}
\newcommand\R{{\mathbb R}}
\newcommand\Qb{{\mathbb Q}}
\newcommand\Z{{\mathbb Z}}

\newcommand\Cc{{\mathcal C}}

\newcommand\Hc{{\mathcal H}}

\newcommand\Pc{{\mathcal P}}

\newcommand\Piecewise[2][]{ \left\{ \begin{array}{ll} #2 \ifthenelse{\equal{#1}{}}{}{\\ #1 & \text{otherwise}}\end{array} \right. }

\providecommand\ding{} \renewcommand\ding[1]{{\scriptsize\fcolorbox{red}{red!10}{\bf\color{red}\!#1\!}\!}}
\def\commentcolor{red}
\newcounter{commentcounter} 
\newcommand{\COMMENT}[2][\commentcolor]{\stepcounter{commentcounter}{\normalsize\textcolor{#1} {\smash{\rlap{\raisebox{0.6em}{\scriptsize\ding\thecommentcounter}}}}}\marginpar{\color{#1}\footnotesize\vskip-.6 \baselineskip\raggedright\hsize1.1in{\normalsize\ding{\thecommentcounter}} \thinspace#2\vskip.3\baselineskip}%
\ifthenelse{\thecommentcounter=191}{\setcounter{commentcounter}{171}}{}%
\ifthenelse{\thecommentcounter=181}{\setcounter{commentcounter}{181}\def\commentcolor{red!50!blue}}{}}

\newcommand\ExtAlpha{{\mathcal A}} 
\providecommand\temp{} 

\title[Finite partitions for complex continued fractions]{Finite partitions for several complex \linebreak continued fraction algorithms}
\author{Adam Abrams}
\address{Institute of Mathematics, Polish Academy of Sciences, {\'S}niadeckich 8, Warsaw, Poland 00656}
\email{the.adam.abrams@gmail.com}
\date{23 October 2019}


\begin{document}
\begin{abstract}We present a property satisfied by a large variety of complex continued fraction algorithms (the ``finite building property'') and use it to explore the structure of bijectivity domains for natural extensions of Gauss maps. Specifically, we show that these domains can each be given as a finite union of Cartesian products in $\C \times \C$. In one complex coordinate, the sets come from explicit manipulation of the continued fraction algorithm, while in the other coordinate the sets are determined by experimental means.\end{abstract}
\maketitle

\section{Introduction} \label{sec intro}

Real continued fractions can be used to study geodesic flow on the modular surface $\Hc^2 \backslash \mathrm{PSL}(2,\Z)$, initially investigated by Artin~\cite{Artin} with further development by Caroline Series \cite{S81, S85} and Adler and Flatto~\cite{AF82, AF84}. Katok and Ugarcovici~\cite{KU05,KU10,KU12} detailed a two-parameter family of algorithms, the so-called $(a,b)$-con\-tin\-ued fraction algorithms, which have applications in both number theory and dynamics.
The main result of~\cite{KU10} is that for $(a,b)$-con\-tin\-ued fraction algorithms, the natural extension of the Gauss map has an attractor in $\R \times \R$ that has ``finite rectangular structure.'' A key tool in Katok--Ugarcovici's analysis is the ``cycle structure.''

Continued fractions for complex numbers have been studied from a number theoretic perspective by Adolf Hurwitz~\cite{AHurwitz}, Doug Hensley~\cite{Hensley}, and more recently by S.~G.~Dani and Arnaldo Nogueira~\cite{DN14}.

The natural extensions of Gauss maps for several real and complex continued fraction algorithms have been used to derive absolutely continuous invariant measures for the Gauss maps themselves. This method was applied by Nakada et al.~\cite{NIT} to the real nearest integer algorithm and its dual (or ``backwards'') algorithm; the minus version of these two algorithms are $(a,b)$-con\-tin\-ued fractions with $(a,b)=(-\tfrac12,\tfrac12)$ and $(a,b)=(\tfrac{1-\sqrt5}2,\tfrac{3-\sqrt5}2)$, respectively. In \cite{KU12}, Katok--Ugarcovici use the natural extension to calculate the invariant measure for any $(a,b)$-con\-tin\-ued fraction Gauss map. In the complex setting, this method was applied by Tanaka \cite{T85} to the nearest even integer algorithm and by Ei et al.~\cite{EINN} to the nearest integer algorithm.

\medskip
In this paper, we investigate complex continued fractions and their Gauss maps' natural extensions $\hat G$ acting on $\C \times \C$. In particular, we describe a substitute for the ``cycle structure'' that can be used with complex algorithms, and we show that a bijectivity region for $\hat G$ can have the form $\bigcup_{i=1}^N K_i \times L_i$, which we call \emph{finite product structure}. When dealing with these Cartesian products, the sets $K_i$ come from explicit manipulation of the continued fraction algorithm while the sets $L_i$ are approximated numerically and, for some algorithms, exact descriptions of these sets are also rigorously proven.

The paper is organized as follows. In \Cref{sec cf}, we present some background on complex continued fractions and prove some technical results. In \Cref{sec fb}, we define the finite building property, give several sufficient conditions for the property to hold, and discuss the finite product structure of sets in $\C^2$. In \Cref{sec examples}, we present six different algorithms in detail (these are also shown in \Cref{fig regions} in \Cref{sec cf}). Partitions for the finite building property are given for all six algorithms, and the finite product structure is given explicitly for some of the algorithms.

\section{Continued fractions} \label{sec cf}

A \emph{minus complex continued fraction} is an expression of the form
\[ \label{cf-in-general}
	a_0 - \dfrac{1}{
	a_1-\dfrac{1}{a_2-\dfrac{1}{\ddots
	\raisebox{-.4em}{$-\;\tfrac{1}{a_n}$}
	}}}
	\qquad\text{or}\qquad
	a_0 - \dfrac{1}{
	a_1-\dfrac{1}{a_2-\dfrac{1}{\ddots
	}}}
\]
where each $a_n$ is a Gaussian integer, that is, an element of $\Z[\i] = \setform{ x + y\i }{ x,y \in \Z }$. The $a_i$ are called the \emph{digits} of the continued fraction (in some works they are called ``partial quotients''\!).
For applications to dynamical systems, finite continued fractions are often ignored, and instead the sole focus is on infinite continued fractions.

Given a sequence $\{a_n\}$, one can define sequences $\{p_n\}$ and $\{q_n\}$ by
\< \label{pq}
	\begin{array}{l@{\qquad\qquad}l@{\qquad\qquad}l}
		p_{-2} = 0	&   p_{-1} = 1   &   p_n = a_n p_{n-1} - p_{n-2}
		\quad\text{for } n \ge 0, \\
		q_{-2} = -1   &   q_{-1} = 0   &   q_n = a_n q_{n-1} - q_{n-2}
		\quad\,\text{for } n \ge 0.
	\end{array}
\>
Formal algebraic manipulations (in any field, not just $\C$) give that
\< \label{convergents}
	\dfrac{p_n}{q_n} = 
	a_0 - \dfrac{1}{
	a_1-\dfrac{1}{a_2-\dfrac{1}{\raisebox{-0.4em}{\smash{
	\raisebox{0.3em}{$\ddots$}
	\raisebox{0em}{$-\;\tfrac{1}{a_n}$}
	}}}}}
\>
assuming $a_n \ne 0$ and $q_n \ne 0$; see \cite[Theorem~2.2]{DN14} for a sufficient condition to imply $q_n \ne 0~\forall~n$. When it exists, the term ${p_n}/{q_n}$ is called the \emph{$n^\text{th}$ convergent} of the continued fraction.


\subsection{Choice functions and algorithms}

Given a value $x \in \R$ or $x \in \C$, there are various algorithms that can be used to construct a finite or infinite sequence $(a_0, a_1, \ldots)$ such that $\smash{ a_0 - \frac1{a_1-\frac1{\smash{\raisebox{-0.35em}{\scriptsize$\ddots$}}}} }$ converges to~$x$ (or equals~$x$ at some point if the sequence is finite). 

The following combines terminology from Dani--Nogueira \cite{DN14} and notation from Katok--Ugar\-co\-vici \cite{KU10, KU12}.\footnote{\,Dani and Nogueira denote a choice function by $f(x)$. Katok and Ugarcovici use the notation $\bra{x}_{a,b}$ for their ``generalized integer part'' function.}

\begin{defn}
A \emph{choice function} is a function
\[ \bra{\cdot}: \C\setminus\{0\} \to \Z[\i] \]
such that $\big\lvert\, z - \bra{z} \big\rvert \le 1$ for all $z \in \C$.
Each choice function has a \emph{fundamental set}~$K$ given by
\< \label{K} K := \overline{\setform{z - \bra{z}}{z \in \C}}. \>
\end{defn}

Throughout this paper, $\D$ is the open unit disk in $\C$ and $\bar\D$ is the closed unit disk. Note that $K \subset \bar\D$ for any choice function.

The most classical example of a choice function is the nearest integer algorithm, also called the ``Hurwitz algorithm,'' in which $\bra{z}$ is the Gaussian integer closest to $z$ and $K$ is the unit square centered at the origin; see \Cref{fig regions}(a).

\begin{remark} The nearest integer algorithm possesses several properties that are not generally required of choice functions. Several other choice functions are shown graphically in \Cref{fig regions} and described in \Cref{sec examples}.
	\begin{itemize}
		\item For the nearest integer algorithm, $K \subset B(0,r)$ for some $r<1$. This property is assumed for certain number theoretic results (for example, \cite[Prop.~2.4]{DN14}). None of the algorithms discussed in this paper other than the nearest integer satisfy this property.
		\item Translates of $K$ tile the complex plane for many algorithms but not for the ``diamond'' and ``disk'' algorithms (see \Cref{ex diamond,ex disk}).
		\item The set $\setform{ z \in \C }{ \bra{z} = 0 }$ coincides with $K$ for every algorithm discussed in this paper \emph{except} the ``nearest odd'' algorithm (\Cref{ex NO}), for which $\bra{z}$ is never $0$.
		\item The set $K$ contains a neighborhood of the origin for all algorithms discussed here \emph{except} for the ``shifted Hurwitz'' algorithm (\Cref{ex SH}).
	\end{itemize}
\end{remark}

Equation \eqref{K} defines the fundamental set for a given choice function. The following \namecref{make choice} shows that the reverse construction is sometimes possible, that is, certain sets in $\C$ can be used to construct choice functions.

\begin{figure}
    \begin{tabular}{cc}
    	\includegraphics[width=0.42\textwidth]{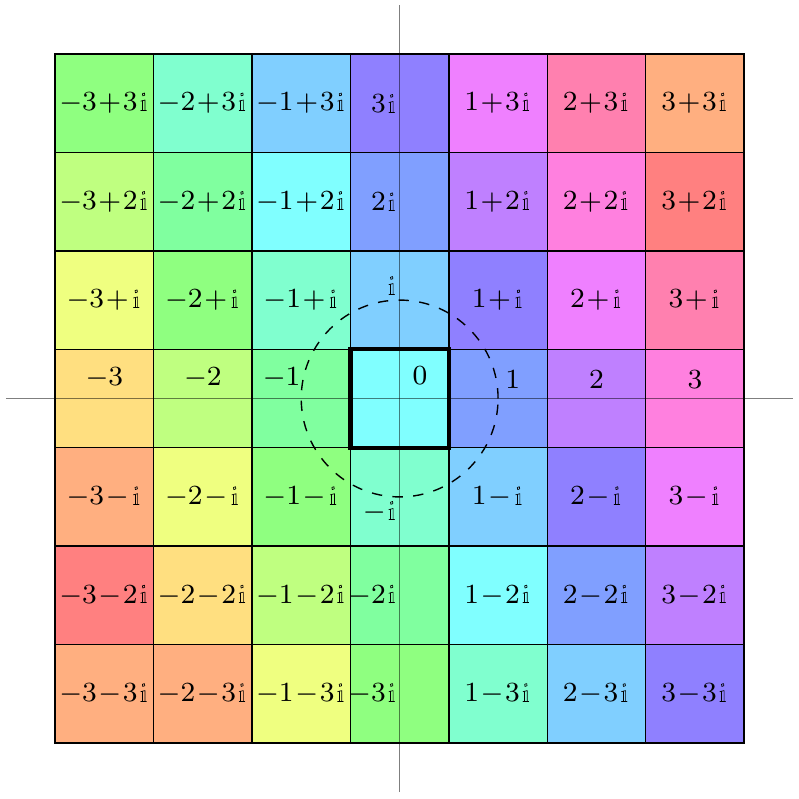} &
    	\includegraphics[width=0.42\textwidth]{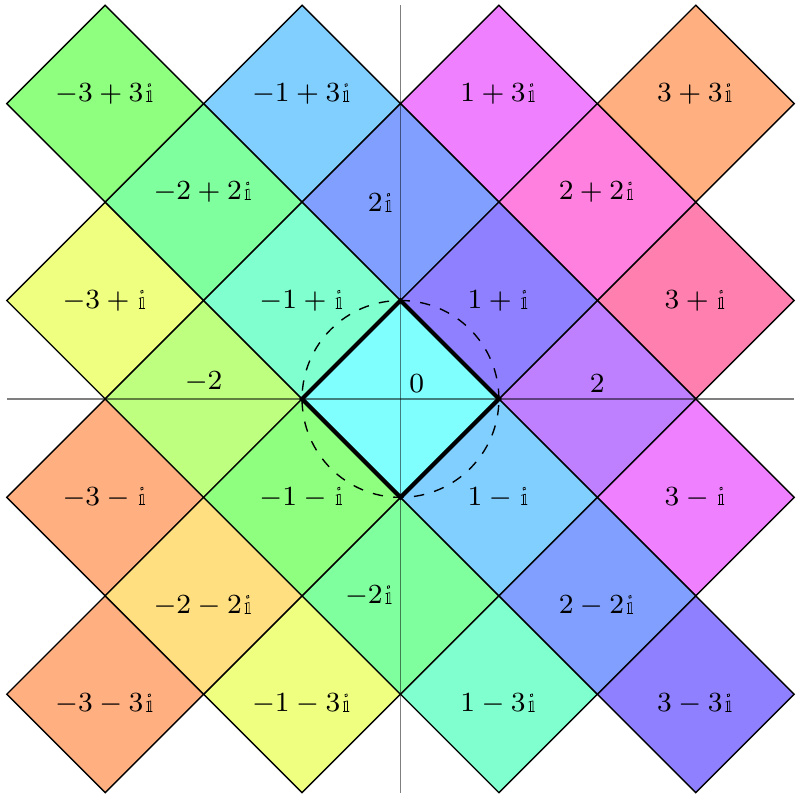} \\[-0.3em]
		(a) Nearest integer (Hurwitz) & (b) Nearest even \\[0.5em]
    	\includegraphics[width=0.42\textwidth]{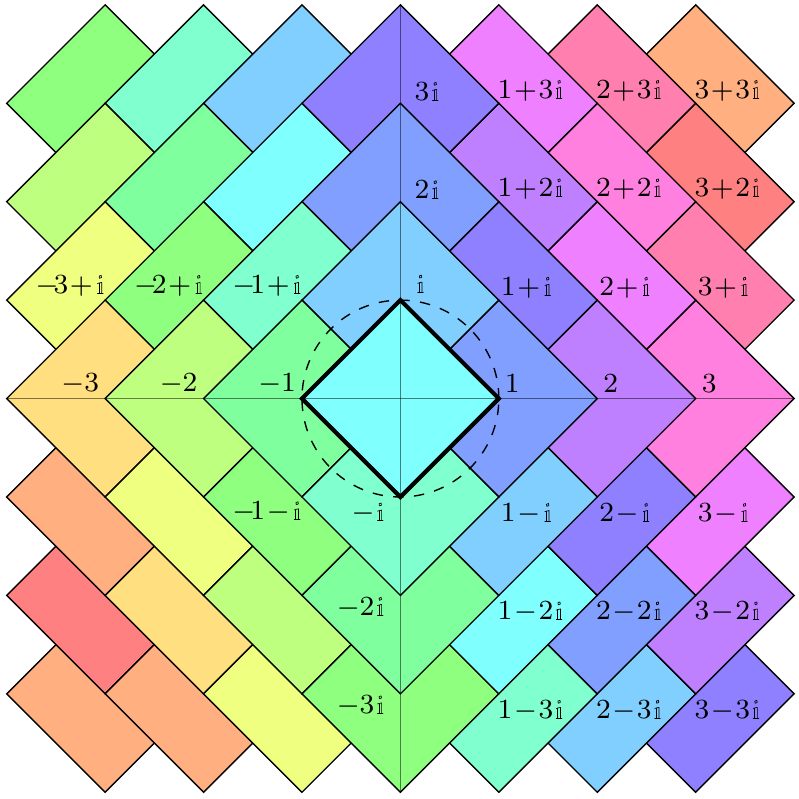} &
    	\includegraphics[width=0.42\textwidth]{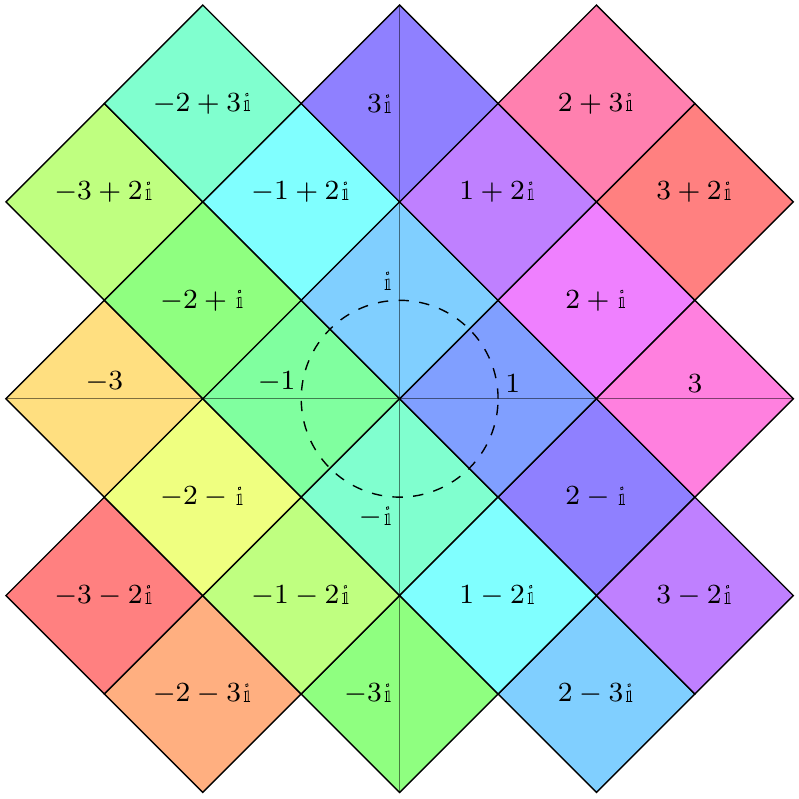} \\[-0.3em]
		(c) Diamond & (d) Nearest odd \\[0.5em]
    	\includegraphics[width=0.42\textwidth]{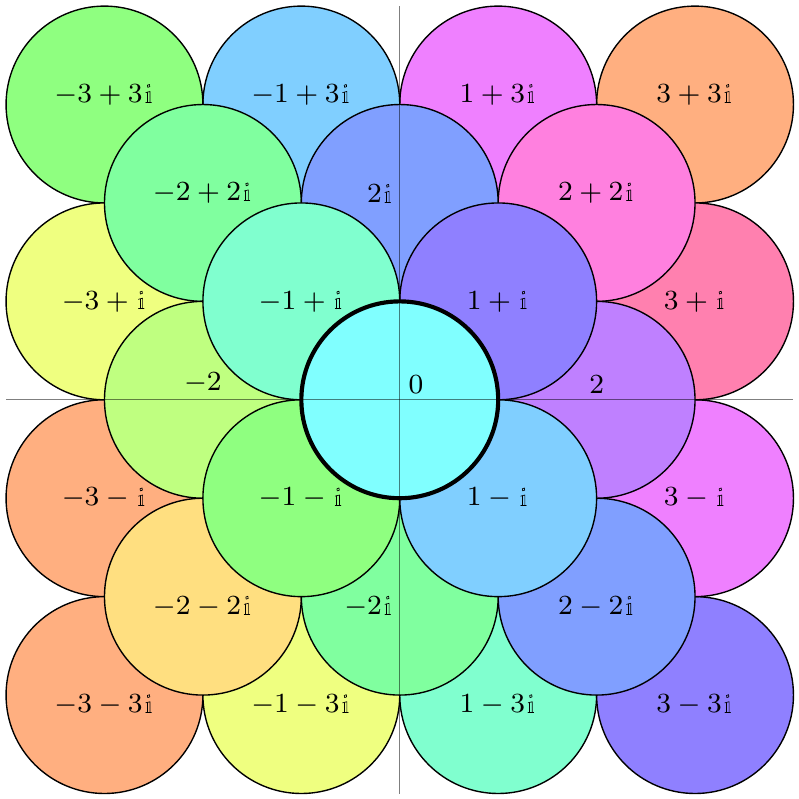} &
    	\includegraphics[width=0.42\textwidth]{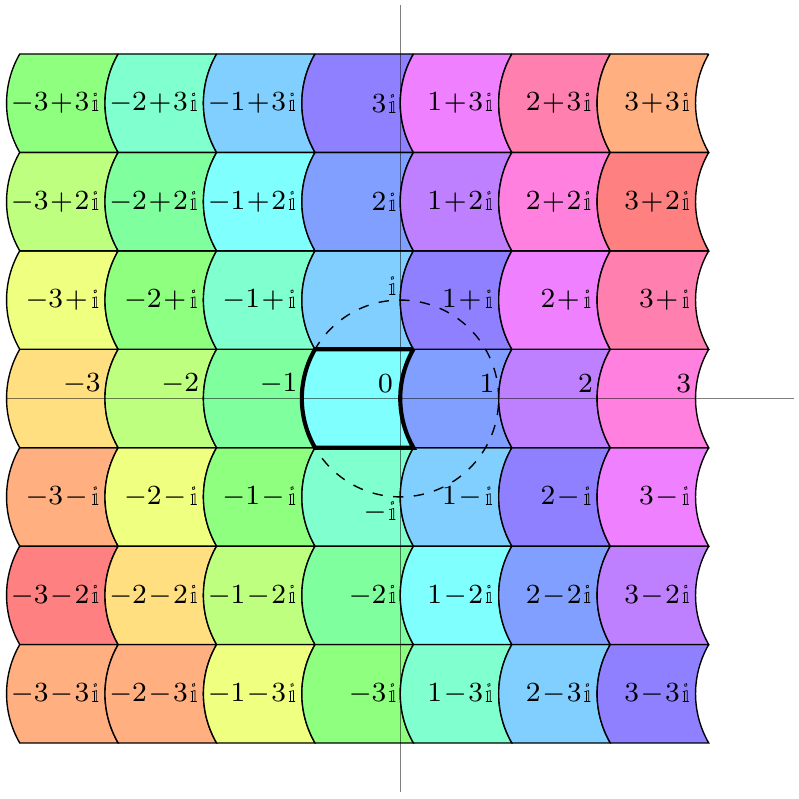} \\[-0.3em]
		(e) Disk & (f) Shifted Hurwitz \\[-0.5em]
	\end{tabular}
    \caption{Regions where $\bra{\cdot}$ takes different values for various algorithms}
    \label{fig regions}
\end{figure}

\renewcommand\temp[1]{\bra{#1}}
\begin{prop} \label{make choice}
Let $X \subset \bar\D$.
If for any $z \in \C \setminus X$ there exists $n(z) \in \N$ such that $t^{n(z)}(z) \in X$, where $t:\C\to\C$ is given by
\[ t(z) = \Piecewise{
	0 & \text{if } z = 0 \\
	z - 1 & \text{if } -\pi/4 \le \arg z < \pi/4 \\
	z - \i & \text{if } \pi/4 \le \arg z < 3\pi/4 \\
	z + 1 & \text{if } 3\pi/4 \le \arg z \text{ or } \arg z < \pi/4 \\
	z + \i & \text{if } -3\pi/4 \le \arg z < -\pi/4,
} \]
then the function 
\[ \temp{z} = \Piecewise{ 0 & \text{if }z \in X \\ z - t^{n(z)}(z) & \text{if }z \notin X } \]
will be a valid choice function.
\end{prop}

\begin{proof}
To prove that $\temp{z}$ is a valid choice function, we need to show that $\temp{z} \in \Z[\i]$ and that $\big\lvert z-\temp{z} \!\big\rvert \le 1$ for all $z \in \C$. 

Since $X \subset \bar\D$, the fact that $\temp{z} = 0$ for $z \in X$ means that both conditions are satisfied for all $z \in X$.

For $z \notin X$, we prove by induction on $n(z)$. 
If $n(z) = 1$, then $z = t(z) + \i^k$ for some $k \in \{0,1,2,3\}$, and so
$ \temp{z} = z - t(z) = \i^k $
is in $\Z[\i]$ and 
$ \big\lvert z-\temp{z} \!\big\rvert = \abs{t(z)} 
\le 1$ because $t(z) \in X \subset \bar\D$.
For $n(z) > 1$, we use that $n(t(z)) = n(z) - 1$ repeatedly until we reach $n({t^{n(z)-1}(z)}) = 1$ and by induction recover that $\temp{z} \in \Z[\i]$ and $\big\lvert z-\temp{z} \!\big\rvert \le 1$ for any $n(z)$.
Thus $\temp{z}$ is a valid choice function.
\end{proof}

If Gaussian integer translates of $K \subset \bar\D$ tile the complex plane without overlap (except possibly on the boundaries of translates of $K$), then~$K$ satisfies the condition of \Cref{make choice}, and the choice function obtained by this construction will be equivalent to 
\[ \label{tile choice} \bra{z} = a \quad\Longleftrightarrow\quad z \in a + K, \quad a \in \Z[\i] \]
except possibly when $z \in \partial(a+K)$.
The nearest integer and shifted Hurwitz algorithms are examples of such integer tilings.

For any choice function $\bra{\cdot}: \C\setminus\{0\} \to \Z[\i]$, the associated \emph{Gauss map} $G: K\to K$ is given by
\< \label{Gauss} G(z) = \frac{-1}z - \bra{ \frac{-1}z } \>
for $z \ne 0$ and $G(0) = 0$.
The Gauss map is ``piecewise continuous'' in the sense that it is continuous on each set 
\< \label{cell} \ang a := \setform{ x \in K }{ \bra{-1/x} = a }. \>

Given $z \in \C$, we can construct the digit sequence $\{a_n\}$ by 
\< \label{a} a_0 = \bra{z}\qquad a_n = \bra{-1/G^{n-1}(z-a_0)}\quad\forall~n \ge 1. \>
After then defining $\{p_n\}$ and $\{q_n\}$ by \eqref{pq}, the sequence $\{\frac{p_n}{q_n}\}$ will either terminate (in which case the final term $\frac{p_n}{q_n} = z$) or will converge to $z$. Terminating sequences only occur for $z \in \Qb[\i]$, see \cite{DN14}, so we will focus only on $z \in \C \setminus \Qb[\i]$.

\begin{defn} \label{even odd}
	A complex number is called \emph{rational} if it is in $\Qb[\i]$ and \emph{irrational} otherwise. A complex number is called \emph{even} (respectively, \emph{odd}) if its real and imaginary parts are both integers and their sum is even (respectively, odd).
\end{defn}

\FloatBarrier
\subsection{Natural extension of the Gauss map}

Defining the notation 
\< \begin{split} \label{S and T}
	S(z) &= -1/z \\
	T^a(z) &= z + a \qquad\text{for any }a\in \Z[\i],
\end{split} \>
we can write the Gauss map $G$ for any choice function $\bra{\cdot}$ as 
\[ G(z) = T^{-a}Sz, \qquad a = \bra{Sz}. \]

Let $\Delta = \setform{ (z,w) \in \C^2 }{ z = w }$.
The natural extension of $G$ is the map ${\hat G} : \C^2 \setminus \Delta \to \C^2 \setminus \Delta$ given by
\begin{equation} \label{natural extension}
    {\hat G}(z,w) = (T^{-a}Sz,T^{-a}Sw), \qquad a = \bra{Sz}.
\end{equation}
In coordinates $(x,y) = (z,Sw)$ or $(u,v) = (Sw,Sz)$ this transformation is given by
\[ \tag{\ref{natural extension}$'$} \begin{split}
    \hat F(x,y) &= (T^{-a}Sx, ST^{-a}y), \qquad a = \bra{Sx} \\
    \hat R(u,v) &= (ST^{-a}u, ST^{-a}v), \qquad a = \bra{v}.
\end{split} \]

\begin{prop} \label{restrict to disk}
	Let $z,w \in \C$ with $z \in K \setminus \Qb[\i]$ and $z \ne w$. 
	Then there exists $n < \infty$ such that ${\hat G}^n(z,w) \in K \times (\C \setminus \bar\D)$.
\end{prop}

\begin{proof}
	Use the coordinates $(u,v) = (Sw,Sz)$. We want to prove that ${\hat R}^n(u,v)$ is in $\D \times S(K)$.
	Note that $z \ne w$ means $u \ne v$, and $z \notin \Qb[\i]$ means $v = -1/z$ is also irrational.
	
	Let $(u_k,v_k) = {\hat R}^k(u,v)$, that is,
    \begin{align*}
    	u_{k+1} &= S T^{-a_k} \cdots S T^{-a_1} S T^{-a_0} u, \\
    	v_{k+1} &= S T^{-a_k} \cdots S T^{-a_1} S T^{-a_0} v,
    \end{align*}
	where $a_k = \bra{v_k}$. Because $v$ is irrational, these sequences do not terminate.
	By construction, all $v_n \in S(K)$, so $(u_n,v_n) \in \D \times S(K)$ is equivalent to $\abs{u_n} < 1$.
	
    For all $k \ge 1$ we have
    \begin{align*}
    	u &= T^{a_0} S T^{a_1} S \cdots T^{a_k} S(u_{k+1}) 
    	= \frac{p_k u_{k+1} - p_{k-1}}{q_k u_{k+1} - q_{k-1}},
    \end{align*}
    where $p_k/q_k$ are the convergents of $v$,
    and thus
    \[
    	u_{k+1} = \frac{q_{k-1} u - p_{k-1}}{q_k u - p_k} 
    	 = \frac{q_{k-1}}{q_k} + \frac1{q_k^2(\frac{p_k}{q_k} - u)}.
	 \]
	 Let $C_k = \frac1{p_k/q_k - u}$. Because $\frac{p_k}{q_k} \to v \ne u$, the sequence $\{C_k\}$ converges to the finite value $\frac1{v-u}$. 
	 Applying the Triangle Inequality to $u_{k+1} = \frac{q_{k-1}}{q_k} + \frac{C_k}{q_k^2}$ gives
	 \begin{align*}
	 	\abs{u_{k+1}}
    	\le \abs{\frac{q_{k-1}}{q_k}} + \abs{\frac{C_k}{q_k^2}} 
    	= \frac{\abs{q_{k-1}} + \frac{\abs{C_k}}{\abs{q_k}}}{\abs{q_k}} 
		= 1 - \frac{\abs{q_k} - \abs{q_{k-1}} - \frac{\abs{C_k}}{\abs{q_k}}}{\abs{q_k}}.
	\end{align*}
	Therefore $\abs{u_{k+1}} < 1$ is implied by
	\[
		\frac{\abs{q_k} - \abs{q_{k-1}} - \frac{\abs{C_k}}{\abs{q_k}}}{\abs{q_k}} > 0 
	\]
	or, equivalently, by
	\begin{equation} \label{good for lemma}
		\abs{q_k} - \abs{q_{k-1}} > \frac{\abs{C_k}}{\abs{q_k}}.
	\end{equation}
	By \Cref{diff of norms} below, \eqref{good for lemma} is true whenever $\abs{q_k} > \abs{q_{k-1}} > (\abs{C_k}+1)\sqrt2$. Since $\{C_k\}$ converges to a finite value, there exists~$M \in \R$ such that $\abs{C_k} < M$ for all $k$. Since $\{q_k\}$ is unbounded, there exists~$n \in \N$ for which \[ \abs{q_{n+1}} > \abs{q_n} > (M+1)\sqrt2 > (\abs{C_{n-1}}+1)\sqrt2, \] and for this $n$ we have that $\abs{u_n}<1$ and therefore ${\hat R}^n(u,v) \in \D \times S(K)$.
\end{proof}

\begin{lem} \label{diff of norms}
	Let $u,v \in \Z[\i]$ and $C \ge 0$. If $\abs{u} > \abs{v} > (C+1)\sqrt2$, then $\abs{u} - \abs{v} > \frac{C}{\abs{u}}$.
\end{lem}

\begin{proof}
\providecommand\Frac[2]{\frac{#1}{#2}}
	Assume by symmetry that $\Re u \ge \Im u \ge 0$, and denote $u_1 = \Re u$, $u_2 = \Im u$. Then $\abs u \le u_1 \sqrt 2$. Combining this with the assumption $\abs u > (C+1)\sqrt2$, we have
	\[ u_1 \sqrt2 ~\ge~ \abs u ~>~ (C+1) \sqrt2, \]
	and so $1 < u_1 - C$. 
	Therefore \providecommand\temp{} \renewcommand\temp{\abs{u}}
	\begin{align}
    	1 &< 2(u_1 - C) + {C^2}/{\temp^2} \nonumber \\
    	\temp^2 - 2u_1 + 1 &< \temp^2 - 2C + \Frac{C^2}{\temp^2} \nonumber \\
    	\sqrt{\temp^2 - 2u_1 + 1} &< \temp-\Frac C\temp. \label{final ineq}
	\end{align}
	Since $u = u_1 + u_2 \i$ with $u_1 \ge u_2$, the largest possible norm of $v \in \Z[\i]$ with $\abs v < \abs u$ is given by
	\begin{align*}
		\sqrt{ (u_1-1)^2 + u_2^2 } \,
		&= \sqrt{ (u_1^2-2u_1+1) - u_1^2 + u_1^2 + u_2^2 } \\
		&= \sqrt{ {\abs u}^2 - 2 u_1 + 1 } \\
		&< \abs u - \frac{C}{\abs u} \qquad \text{by \eqref{final ineq}.}
	\end{align*}
	Therefore $\abs v < \abs u - \frac{C}{\abs u}$, or, equivalently, $\abs u - \abs v > \frac{C}{\abs u}$.
\end{proof}


\section{The finite building property} \label{sec fb}

In the case of real $(a,b)$-con\-tin\-ued fractions, the orbits of the two discontinuity points $a$ and $b$ of the map
\[ f_{a,b}(x) = \left\{ \begin{array}{ll} x+1 &\text{if } x < a \\ -1/x &\text{if } a \le x < b \\ x-1 &\text{if } x \ge b \end{array}\right. \label{fab} \]
collide after finitely many iterations, and this ``cycle property'' is heavily used in the analysis of the real-valued Gauss map and its natural extension in \cite{KU10,KU12}. In the complex setting, $K \subset \C$ replaces the interval $[a,b)$, but since $\partial K$ is not a finite set of points, tracking its orbit is significantly more complicated. The ``finite building property'' described in \Cref{defn partition property} serves as a replacement for the cycle property.

\begin{defn} \label{defn buildable}
	Let $\Cc$ be a collection of closed sets whose boundaries each have zero measure.
	A set is called \emph{buildable} from $\Cc$ if it is equal, up to measure zero, to some union of elements of $\Cc$.
\end{defn}
\begin{defn} \label{defn partition property}
	A continued fraction algorithm with Gauss map $G : K \to K$ satisfies the \emph{finite building property} if there exists a finite partition $\Pc = \{ K_1, \ldots, K_N \}$ of $K$ with $N>1$ such that each $G(K_i)$ is buildable from $\Pc$.
\end{defn}

\begin{remark}
	The term \emph{partition} here means that the \emph{interiors} of $K_i$ and $K_j$ must be disjoint for $i\ne j$. In some works, partition elements must be truly disjoint, but Markov partitions are ``partitions'' in exactly this sense.
\end{remark}

There are some ``shortcuts'' one may use to prove that an algorithm satisfies the finite building property without directly testing \Cref{defn partition property}. The following statements give sufficient conditions for an algorithm to satisfy the finite building property.

\begin{prop} \label{pre tiling shortcut}
	Let $\Pc = \{K_1,\ldots,K_N\}$ be a partition of $K$. If each $S(K_i)$ can be written as a union $\bigcup_{\alpha} W_\alpha$ such that $\bra{\cdot}$ is constant on each $W_\alpha$ and each $W_\alpha = \bra{W_\alpha} + K_j$ for some $K_j \in \Pc$, then the associated continued fraction algorithm satisfies the finite building property.
\end{prop}

\begin{proof}
	\def\temp{A}
	Fix $i$, $1\le i\le N$, and let $\temp \subset \Z[\i] \times \{ 1,\ldots,N \}$ be the index set for $\alpha$, that is,
	\[ S(K_i) = \bigcup_{(a,j) \in \temp} a + K_j \]
	with $\bra{a + K_j} = a$ for $(a,j) \in \temp$.
	Then we immediately have that
	\[ G(K_i) =  \bigcup_{(a,j) \in \temp} T^{-a}(a + K_j) = \bigcup_{(a,j) \in \temp} K_j \]
	is buildable from $\Pc$.
\end{proof}

\begin{cor} \label{tiling shortcut}
	Suppose $K$ is equal to the closure of its interior and that there exists a set $Z \subseteq \Z[\i]$ such that if $a \in Z$ and $w$ is in the interior of $a + K$ then $\bra w = a$.
	 
	Let $\{K_1,\ldots,K_N\}$ be a partition of $K$. If each $S(K_i)$ can be written as a union of sets of the form $a + K_j$ with $a \in Z$, then the associated continued fraction algorithm satisfies the finite building property.
\end{cor}

Note that the conditions of \Cref{tiling shortcut} are satisfied for the nearest integer, nearest even, nearest odd, and shifted Hurwitz algorithms, all of which can have translates of $K$ tile the complex plane.

\begin{prop} \label{Markov shortcut} Let $\{K_1,\ldots,K_N\}$ be a partition of $K$ and recall the notation $\ang a$ from \eqref{cell}. If 
\begin{enumerate}
	\item for each $a \in \Z[\i]$ there is an $i$ such that $\ang a \subset K_i$, and 
	\item the set $S(\ang a)$ can be written as a union of sets of the form $a + K_j$,
\end{enumerate}
then the associated continued fraction algorithm satisfies the finite building property. \end{prop}

\begin{proof}
	Unlike \Cref{pre tiling shortcut}, we don't need to state any conditions on $\bra{\cdot}$. We need only that $G(z) = T^{-a}S(z)$ for all $z \in \ang{a}$, which is true by \eqref{cell}.

	For each $1 \le i \le N$, let 
	\[ A(i) = \setform{ a \in \Z[\i] }{ \ang a \subset K_i } \]
	and let $J(a) \subset \{1,\ldots,N\}$ be such that $S(\ang a) = \bigcup_{j \in J(a)} (a+K_j)$.
	Then 
	\begin{align*}
    	G(K_i) 
    	&= G\!\left( \bigcup_{a \in A(i)} \ang a \right)
    	= \bigcup_{a \in A(i)} G(\ang a) 
    	= \bigcup_{a \in A(i)} T^{-a} S \ang a 
		= \bigcup_{a \in A(i)} \bigcup_{j \in J(a)} K_j
    \end{align*}
    and so $G(K_i)$ is buildable from $\Pc$.
\end{proof}

In general, each $\ang a$ might not be a subset of any $K_i$, and so we will have to look at multiple intersections $\ang a \cap K_i$. Some new notation will be helpful.

\begin{defn} \label{defn new notations}
	Fix a partition $\{K_1,\ldots,K_N\}$.
	Define
	\< \label{new alphabet} \ExtAlpha = \setform{ (a,i) }{ a \in \Z[\i], 1 \le i \le N, \ang a \cap K_i \ne \emptyset }\!. \>
	For any $a \in \Z[\i]$ and $1 \le i \le N$, denote
	\<
		K_{i,a} 
		= K_i \cap \ang a
		= \setform{ x \in K_i }{ \bra{S x} = a }\!.
	\>
	
	Lastly, for each $1 \le i \le N$ we define $\ExtAlpha_i \subset \ExtAlpha$ by  
    \< \label{ExtAlpha} \begin{split}
    	\ExtAlpha_i &= \setform{ (a,j) }{ K_i \subset G\big(K_{j,a}\big) }
    	\\ &= \setform{ (a,j) }{ K_i \subset T^{-a} S \big(K_{j,a}\big) }
    	\\ &= \setform{ (a,j) }{ S T^a(K_i) \subset K_{j,a} }\!.
    \end{split} \>
\end{defn}

Note that for any algorithm, a partition satisfying the finite building property will not be unique. In \Cref{sec examples} we present partitions $\Pc$ for several algorithms and prove that each satisfies the finite building property. The process described below was used to produce each of these partitions.

\begin{samepage}
\begin{prop}[Partiton creation] \label{partition construction}
	Fix a continued fraction algorithm, and let $\Pc_0 = \{K\}$. Iteratively repeat the following process: 
	\begin{itemize} \item if there exists $a \in \Z[\i]$ and $k \in \Pc_n$ such that the set
	\< \label{not buildable} X = \setform{ T^{-a}Sz }{ \bra{Sz} = a \text{ and } z \in k } \>
	is \emph{not} buildable from $\Pc_n$, then let 
	\[ \Pc_{n+1} = \Pc_n \vee \big\{ X, K \setminus X \big\} \]
	where $\mathcal A \vee \mathcal B = \setform{ A \cap B }{ A \in \mathcal A, B \in \mathcal B }$. 
	\end{itemize}
	If at some finite stage every set of the form \eqref{not buildable} is buildable from $\Pc_n$, then the continued fraction algorithm satisfies the finite building property.
\end{prop}
\end{samepage}

This follows immediately from
\[ G(K_i) = \bigcup_{a \in \Z[\ssi]} \setform{ T^{-a}Sz }{ \bra{Sz} = a \text{ and } z \in K_i }\!, \]
meaning that $G(K_i)$ is a union of sets of the form \eqref{not buildable} (so if all such sets are buildable by $\Pc_n$ then \Cref{defn partition property} is satisfied).

\begin{figure}[hbt]
	\includegraphics{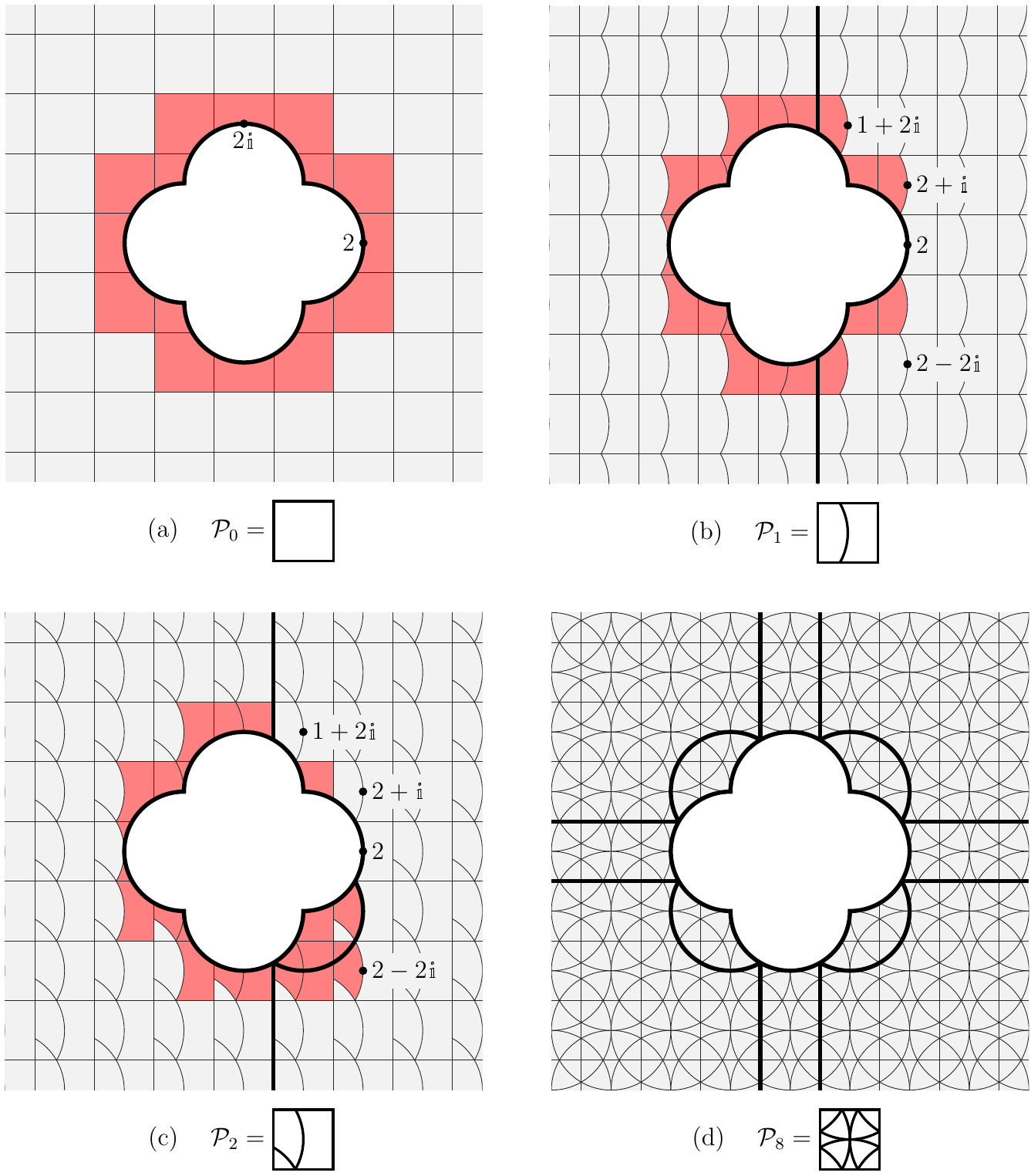}
    \caption{Stages in the process of constructing $\Pc$ for the nearest integer algorithm}
    \label{fig process}
\end{figure}

\Cref{fig process} shows an application of this process for the nearest integer algorithm. In the top-left (``Stage $0$''), the red portions are sets $a+X$ for which $X$ is not buildable from $\Pc_0 = \{K\}$. For example, with $a = 2$ and $k = K$ we have 
\begin{center}
	\begin{tikzpicture}
		\draw (-1,0) node [left] {$X=\setform{z\in K}{\abs{z+1}\ge 1}$};
		
		\draw [fill=black!33] (-0.133974596,1/2) arc (30:-30:1) -- (1/2,-1/2) -- (1/2,1/2) -- cycle;
		\draw [dashed] (-0.133974596,1/2) -- (-1/2,1/2) -- (-1/2,-1/2) -- (-0.133974596,-1/2);
	\end{tikzpicture}
\end{center}
Using this to create $\Pc_1$ gives \Cref{fig process}(b) (``Stage $1$''), which shows thin lines for every integer translation of $\Pc_1$.\footnote{\,Using a different value of $a$ would give a different $\Pc_1$, but in the end some $\Pc_n$ would still be exactly the partition from \eqref{Hurwitz partition}.} The area around $2 \in \C$ is now good, but using $a = 2 + \i$ and $k$ the complement of the previous $X$ gives 
\begin{center}
	\begin{tikzpicture}
		\draw (-1,0) node [left] {$X=\setform{z\in K}{\abs{z+1}\le 1, \abs{z+1+\i} \ge 1}$};
		
		\draw [fill=black!33] (-0.133974596,1/2) arc (30:-30:1) arc (30:60:1) -- (-1/2,1/2) -- cycle;
		\draw [dashed] (-0.133974596,1/2) -- (1/2,1/2) -- (1/2,-1/2) -- (-1/2,-1/2) -- (-1/2,-0.133974596);
	\end{tikzpicture}
\end{center}
which is still red in Stage $1$. This set is used to create $\Pc_2$ in \Cref{fig process}(c). Now all pieces bordering $2$, $2 + \i$, and $1+2\i$ are gray, but note that around $1-2\i$ and $2-2\i$ there are pieces that had been gray in Stages $0$ and $1$ but are now red in Stage $2$. This is because of how the arc of $B(-1-\i,1)$, which is used to form boundaries in $\Pc_2$, intersects $\ang{1-2\i}$ and $\ang{2-2\i}$ (it also intersects $\ang{2-\i}$, but this does not create any unbuildable pieces).

Fortunately, after eight steps the process of \Cref{partition construction} does yield a partition $\Pc_8 = \{K_1,\ldots,K_{12}\}$ such that $G(k)$ is buildable from $\Pc_8$ for every $k \in \Pc_8$. This is precisely the partition given in \eqref{Hurwitz partition} in \Cref{ex Hurwitz}.

\FloatBarrier
\subsection{Finite product structure}

\begin{thm} \label{fps iff system A}
	Consider an algorithm that satisfies the finite building property with partition $\{K_1,\ldots,K_N\}$, and let $L_1, \ldots L_N \subset \C$ be arbitrary closed sets such that the boundaries of each $K_i \times L_i$ have zero $2$-dimension Lebesgue measure. The map ${\hat G}$ is bijective a.e.~on the set
	\< \label{omega A} \Omega := \bigcup_{i=1}^N K_i \times L_i \>
	if and only if the following system holds:
	\< \label{system A} L_i = \bigcup_{(a,j) \in \ExtAlpha_i} T^{-a} S L_j, \qquad 1 \le i \le N. \>
\end{thm}

\begin{lem} \label{fps lemma A}
	For arbitrary sets $L_1,\ldots,L_N \subset \C$, we have that
	\[ {\hat G}\!\left(\, \bigcup_{i=1}^N K_i \times L_i \right) = \bigcup_{i=1}^N \left( K_i \times \left( \bigcup_{(a,j) \in \ExtAlpha_i} T^{-a} S L_j \right) \right). \]
\end{lem}

\begin{proof}
Using $K_{i,a}$ and $\ExtAlpha$ from \Cref{defn new notations}, we can decompose $\Omega$ as
\begin{align*}
    \bigcup_{i=1}^N K_i \times L_i
    &= \bigcup_{i=1}^N \left( \left( \bigcup_{\substack{a \in \Z[\ssi] \\ K_{i,a} \ne \emptyset}} K_{i,a} \right) \times L_i \right) \\
    &= \bigcup_{i=1}^N \bigcup_{\substack{a \in \Z[\ssi] \\ K_{i,a} \ne \emptyset}} K_{i,a} \times L_i \\
    &= \bigcup_{(a,i) \in \ExtAlpha} K_{i,a} \times L_i.
\end{align*}
Then we look at the image of this union under ${\hat G}$.
\begin{align*}
	{\hat G}\!\left(\, \bigcup_{i=1}^N K_i \times L_i \right) 
	&= {\hat G}\!\left( \bigcup_{(a,j) \in \ExtAlpha} K_{j,a} \times  L_j \right) \\
	&= \bigcup_{(a,j) \in \ExtAlpha} \left( T^{-a} S K_{j,a} \times  T^{-a} S L_j \right) \\[-0.5em]
	&= \bigcup_{(a,j) \in \ExtAlpha} \left( \left(\bigcup_{\substack{1 \le i \le N \\ (a,j) \in \ExtAlpha_i}} K_i \right) \times  T^{-a} S L_j \right) \\[0.2em]
	&= \bigcup_{(a,j) \in \ExtAlpha} \; \bigcup_{\substack{1 \le i \le N \\ (a,j) \in \ExtAlpha_i}} \left( K_i \times T^{-a} S L_j \right) \\
	&= \bigcup_{i=1}^N \; \bigcup_{(a,j) \in \ExtAlpha_i} \left( K_i \times T^{-a} S L_j \right) \\
	&= \bigcup_{i=1}^N \left( K_i \times \left( \bigcup_{(a,j) \in \ExtAlpha_i} T^{-a} S L_j \right) \right) 
\end{align*}
where we recall that $(a,j) \in \ExtAlpha_i$ means that $K_i \subset G K_{j,a}$.
\end{proof}

\begin{proof}[Proof of \Cref{fps iff system A}]
Assume \eqref{system A} holds. Then \Cref{fps lemma A} immediately gives that ${\hat G}(\Omega) = \Omega$. 
A function is always surjective onto its image. Because all of the unions here are disjoint except on boundaries (which by assumption have measure zero), and because each transformation $T^{-a} S$ is bijective, ${\hat G}$ is injective except on a set of zero measure. Thus ${\hat G}$ is bijective a.e.~on $\Omega$.

\medskip
Now assume that ${\hat G}$ is bijective almost everywhere on $\Omega$, so ${\hat G}(\Omega)$ must equal $\Omega$ (both are closed). Then \Cref{fps lemma A} gives that
\[
	{\hat G}(\Omega) 
	= \bigcup_{i=1}^N \left( K_i \times \left( \bigcup_{(a,j) \in \ExtAlpha_i}  T^{-a} S L_j \right) \right)
\]
and in order for this to equal $\bigcup_{i=1}^N K_i \times L_i$ it must be that 
\[ \bigcup_{(a,j) \in \ExtAlpha_i} T^{-a} S L_j = L_i \]
for each $1\le i\le N$. This is exactly the system \eqref{system A}.
\end{proof}

The question remains what kind of sets $L_i$ could satisfy \eqref{system A}. Some examples of $L_1,\ldots,L_N$ for specific algorithms are given in \Cref{sec examples}, but in general it is not easy to construct $\{L_i\}$ given $\{K_i\}$.

\Cref{fps iff system A} concerns bijectivity domains of $\hat G$. Ideally, we would like for $\bigcup K_i \times L_i$ to also be an attractor for $\hat G$. The following \namecref{attractor} gives a sufficient, but not necessary, condition for this.

\begin{thm} \label{attractor} 
Assume that for all $z \in K \setminus \Qb[\i]$ the norms $\abs{a_n}$ of continued fraction digits $a_n$ are unbounded.

Let $\Omega = \bigcup_{i=1}^N K_i \times L_i$ be a bijectivity domain for $\hat G$. If each $S(L_i)$ is bounded, then for every $(z,w) \in \C \times \C$ with $z$ irrational there exists $n < \infty$ such that ${\hat G}^n(z,w) \in \Omega$.
\end{thm}

\begin{proof}
	\Cref{restrict to disk} shows there exists $m < \infty$ such that ${\hat G}^m(z,w) \in K \times S(\D)$. Thus we can assume $(z,w) \in K \times S(\D)$.
	
	Suppose $(z,w) \notin \Omega$.
	Let $M \in \N$ be such that each $S(L_i)$ is contained in the ball $B(0,\tfrac1M)$ of radius $\tfrac1M$ centered at the origin in the complex plane. Equivalently, $L_i \subset \C \setminus B(0,M)$ for all $1 \le i \le N$.
	Because $\abs{a_k}$ is unbounded, we can assume $a_1 = \bra{-1/z}$ has absolute value at least $M+2$ (replacing if necessary $z$ by some iterate $G^k z$). Because $w \in S(\D)$, we have $Sw \in \D$ and
	\[ \abs{ T^{-a_1}Sw } > \abs{ T^{-(M+1)}Sw } > M. \]
	This means that
	\[ \hat G(z,w) = (T^{-a_1}Sz, T^{-a_1}Sw) \]
	will be inside $K_j \times \big(\C \setminus B(0,M)\big)$ for some $1 \le j \le N$, and this product is contained in $K_j \times L_j \subset \Omega$ because each $L_j \subset \C \setminus B(0,M)$.
\end{proof}

\subsection{Practical determination of non-leading coordinates} \label{experimental}

For real $(a,b)$-con\-tin\-ued fractions, the analogue of \eqref{system A} is an overdetermined system of equations in $\R$. For specific algorithms one can solve this system exactly and get explicit descriptions of ``$\Lambda_{a,b}$,'' which is analogous to $\Omega$ here.

Using the system \eqref{system A} to ``solve'' for the sets $L_1,\ldots,L_N \subset \C$ given $K_1,\ldots,K_N$ is not practical. For the algorithms discussed in \Cref{sec examples}, sets $K_i$ and $L_i$ are described and then \eqref{system A} is verified to be correct, but a natural question is how these sets were determined in the first place. 
\Cref{partition construction} describes the construction of $\Pc = \{K_1,\ldots,K_N\}$. This process is carried out by hand. Once the $K_i$ are known, the process of finding the corresponding $L_i$ involves computational assistance; this method is most easily described by an example in the real setting.

\medskip
\providecommand\Frac[2]{\tfrac{#1}{#2}}
\providecommand\nFrac[2]{\Frac{-#1}{#2}}

For any $a \le 0 \le b$ satisfying $b-a \ge 1$ and $-ab \le 1$, we define as in \cite{KU10} the maps
\begin{align*}
	\bra{x}_{a,b} &:= \left\{\begin{array}{ll} \floor{x-a} &\text{if } x < a \\ 0 &\text{if } a \le x < b \\ \ceil{x-b} &\text{if } x \ge b \end{array}\right.
	\\[0.5em]
	G_{a,b}(x) &:= \frac{-1}x - \bra{\frac{-1}x}_{a,b}
	\\[0.5em]
	\hat G_{a,b}(x,w) &:= \left( \frac{-1}x - n, \frac{-1}w - n \right), \quad n = \bra{\frac{-1}x}_{a,b},
\end{align*}
and denote by $\Omega_{a,b} \subset \R^2$ the attractor of $\hat G_{a,b}$.\footnote{\,In \cite{KU10}, $G_{a,b}$ is denoted $\hat f_{a,b}$, and the set $\Omega_{a,b}$ is $\setform{(x,-1/y)}{(x,y) \in \smash{\hat\Lambda_{a,b}}}$.}

\medskip
Let $a = -\frac45, b = \frac25$. Then the intervals
\begin{align*}
	K_1 &= [\nFrac45, \nFrac35] &
	K_2 &= [\nFrac35, \nFrac12] &
	K_3 &= [\nFrac12, \nFrac13] \\
	K_4 &= [\nFrac13, \Frac15] &
	K_5 &= [\Frac15, \Frac14] &
	K_6 &= [\Frac14, \Frac25]
\end{align*}
form a partition $\Pc$ of $K = [-\tfrac45,\tfrac25]$ for which every $G_{-4/5,2/5}(K_i)$ is buildable from $\Pc$. There exist intervals $L_1,\ldots,L_6 \subset \R$ such that
\[ \Omega_{-4/5,2/5} = \bigcup_{i=1}^6 K_i \times L_i, \]
and the question is how to find these $L_i$. Instead of using Katok--Ugarcovici's overdetermined system, \Cref{fig nonleading} shows an experimental approach to this problem for $i=4$.

\begin{figure}[ht]
    \includegraphics{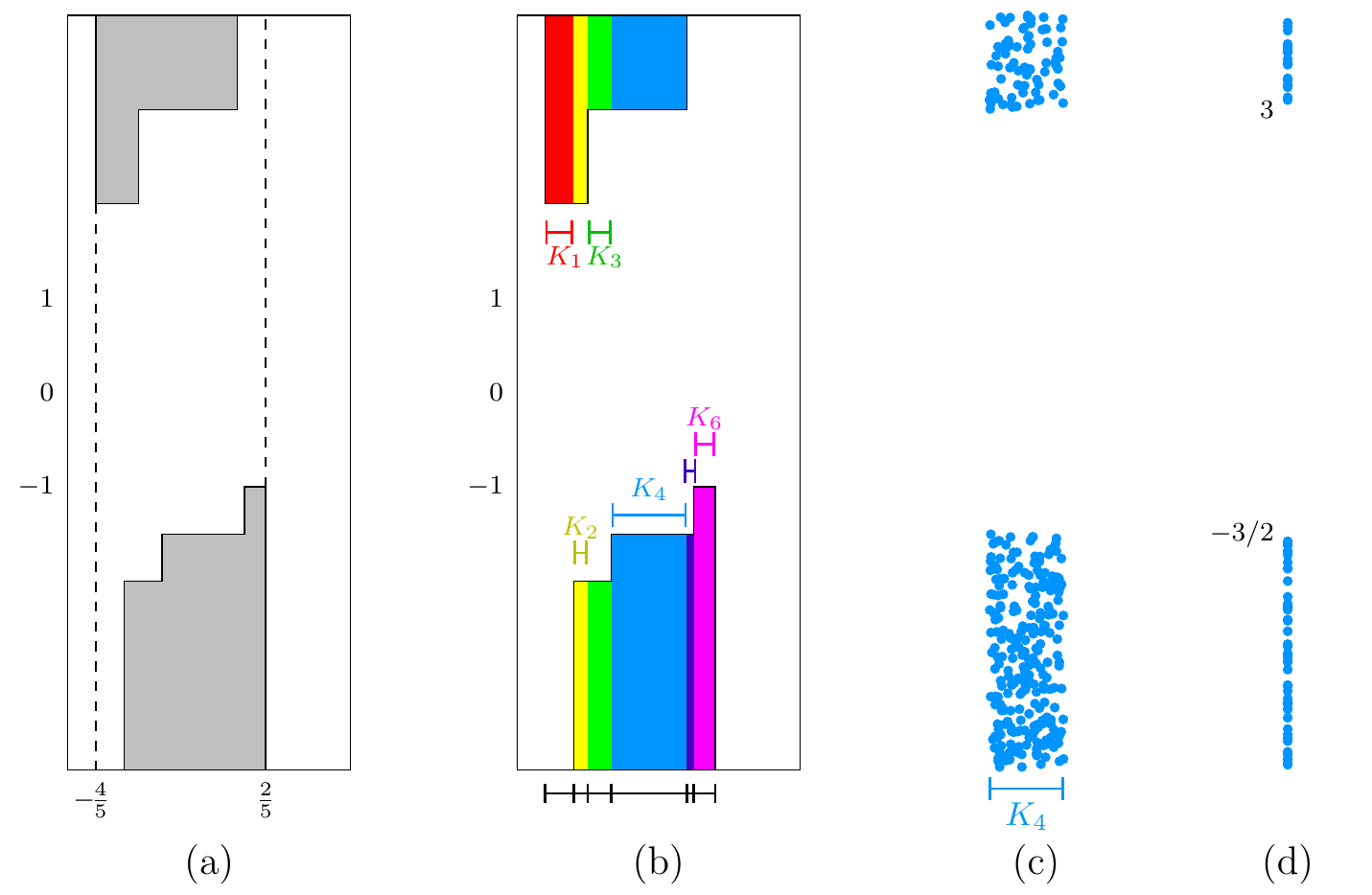}
    \caption{(a) $\Omega_{-4/5,2/5}$. (b) Decomposition into $\bigcup K_i \times L_i$. (c) Numerical plot of points in $\Omega_{-4/5,2/5} \cap (K_4 \times \R)$. (d) Projection of scatter plot onto $y$-axis}
    \label{fig nonleading}
\end{figure}

After a computer iterates random points under $\hat G_{-4/5,2/5}$, it can plot an approximation of 
\begin{align*}
    \mathrm{proj}_2\big( \Omega_{-4/5,2/5} \cap (K_4 \times \R) \big)
    &= \setform{ y }{ \exists~ x \in K_4 \text{ s.t. } (x,y) \in \Omega_{-4/5,2/5} },
\end{align*}
where $\mathrm{proj}_2(x,y) = y$. Such a plot is shown vertically in \Cref{fig nonleading}(d). Visual inspection shows that this scatter plot appears to be $\bar\R \setminus (\nFrac32,3) = \text{``}[3,\nFrac32]\text{''} \subset \R P^1$, so this is our candidate for $L_4$.
Similar observations provide 
\begin{align*}
	L_1 &= [2, \infty] &
	L_2 &= [2, -2] &
	L_3 &= [3, -2] \\
	L_4 &= [3, \nFrac32] &
	L_5 &= [\infty, \nFrac32] &
	L_6 &= [\infty, -1],
\end{align*}
and then one can verify that 
\[ \hat G_{-4/5,2/5} \left( \bigcup_{i=1}^6 K_i \times L_i \right) = \bigcup_{i=1}^6 K_i \times L_i \]
is indeed true for these sets.
 In practice it is easier to plot approximations of $S(L_i)$ because these are bounded in $\R$, e.g., $S(L_4) = [\nFrac13,\Frac23]$.

\medskip
In the complex setting, the process works almost identically. A computer can iterate random points in $\C \times \C$ under $\hat G$ for a given complex continued fraction algorithm and then generate scatter plots approximating a set
\begin{align*}
    \mathrm{proj}_2\big( \Omega \cap (K_i \times \C) \big)
    &= \setform{ w }{ \exists~ z \in K_i \text{ s.t. } (z,w) \in \Omega }
\end{align*}
or its image under $S$. \Cref{fig nonleading complex} shows an approximation (left) of $S L_1$ for the nearest even algorithm---the computer is given the function $\bra{z}$ from \eqref{NE choice} and the sets $K_1,\ldots,K_8$ from \eqref{NE partition}---along with the actual set $SL_1$ (right of \Cref{fig nonleading complex}) as described in the first line of \eqref{NE SL} and shown in \Cref{fig NE products}.

\begin{figure}[ht]
    \includegraphics{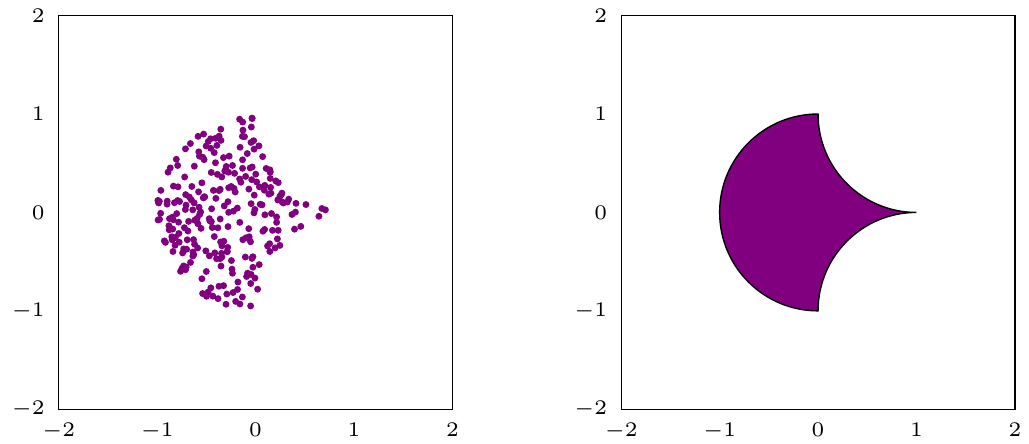}
    \caption{Determining $S(L_1)$ for the nearest even algorithm by approximation}
    \label{fig nonleading complex}
\end{figure}

For the nearest integer algorithm, only experimental scatter plots of $L_i$ are~known; this is how \Cref{fig Hurwitz products} was created, with more detail coming from more iterations.

\FloatBarrier 

\section{Explicit partitions for specific complex algorithms} \label{sec examples}
\newcommand\MyCaption[1]{\caption{Left:~$K_{i,a}$ for the #1 algorithm, colored by~$i$. Right:~image under $S$}}

Here we give explicit descriptions of $K_1,\ldots,K_N$ for various algorithms and prove that each algorithm satisfies the finite building property. Often these proofs make use of \Cref{pre tiling shortcut}, \Cref{tiling shortcut}, or \Cref{Markov shortcut} (see \Cref{tab shortcuts}).

\begin{table}[htb]
	\begin{tabular}{|lccc|} \hline
		& \bf Prop.\,\ref{pre tiling shortcut} & \bf Cor.\,\ref{tiling shortcut} & \bf Prop.\,\ref{Markov shortcut} \\
		\bf Algorithm & \bf applies & \bf applies & \bf applies \\ \hline
		Nearest integer & Yes & Yes, $Z = \Z[\i]$ & No \\
		Nearest even & Yes & Yes, $Z = \text{evens}$ & Yes \\
		Nearest odd & Yes & Yes, $Z = \text{odds}$ & No \\
		Diamond & Yes & No & No \\
		Disk & Yes & No & Yes \\
		Shifted Hurwitz & Yes & Yes, $Z = \Z[\i]$ & No \\ \hline
	\end{tabular}
	\caption{Additional properties of specific algorithms}
	\label{tab shortcuts}
\end{table}

\begin{samepage}
For some of the algorithms, images of the set
$ \bigcup_{i=1}^N K_i \times S(L_i) $
are shown. Note that $S(L_i)$ is used instead of $L_i$ in Figures~\ref{fig Hurwitz products}, \ref{fig NE products}, \ref{fig diamond products}, and \ref{fig disk products} because generally $S(L_i) \subset \bar\D$ and only figures of bounded sets can be shown in full.
\end{samepage}

\FloatBarrier
\subsection{The nearest integer (Hurwitz) algorithm} \label{ex Hurwitz}

The \emph{nearest integer algorithm} assigns to $z \in \C$ the Gaussian integer closest to~$z$. This algorithm was discussed in detail by Adolf Hurwitz~\cite{AHurwitz} and is also called the \emph{Hurwitz algorithm}.\footnote{\,Brothers Adolf and Julius Hurwitz both studied continued fractions. The term ``Hurwitz algorithm'' generally refers to the nearest integer algorithm, while the nearest even algorithm (\Cref{ex NE}) is sometimes called the ``J.~Hurwitz algorithm'' \cite{Oswald}.} The convention to use when $z$ is has multiple closest integers does not have a great effect since the set of all such $z$ has zero measure, but one common convention is to use
\[ \bra{z} = \floor{\Re z + \tfrac12} + \floor{\Im z + \tfrac12}\i. \] 

Partition the the unit square centered at the origin into the following $12$ regions, which are shown in \Cref{fig Hurwitz partition}.
\< \label{Hurwitz partition} \begin{split}
	K_1 &= \setform{ z \in K }{ \Re z \le 0, \abs{z-\i} \ge 1, \abs{z+\i} \ge 1 } \\
	K_2 &= \setform{ z \in K }{ \abs{z-\i} \le 1, \abs{z+1} \le 1, \abs{z-(-1+\i)} \ge 1 } \\
	K_3 &= \setform{ z \in K }{ \abs{z-(-1+\i)} \le 1 } \\
	K_i &= -\i\, K_{i-3} \quad\text{for } i = 4,\ldots,12.
\end{split} \>

\begin{figure}[ht]
    \begin{tikzpicture}
    \begin{scope}[scale=4]
	\foreach \r in {0,90,180,270} \draw [rotate=\r] (0,0) arc (180:150:1) arc (-150:-120:1) arc (-60:-90:1);
	\draw [thick] (-1/2,-1/2) rectangle (1/2,1/2);
	\foreach \k/\r in {1/180, 4/90, 7/0, 10/-90} \draw (\r:0.4) node {\footnotesize$K_{\k}$};
	\foreach \k/\r in {2/135, 5/45, 8/-45, 11/-135} \draw (\r:0.25) node {$K_{\k}$};
	\foreach \k/\r in {3/135, 6/45, 9/-45, 12/-135} \draw (\r:0.52) node {$K_{\k}$};
	\end{scope}
    \end{tikzpicture}
    \caption{Finite partition of $K$ for the nearest integer algorithm}
    \label{fig Hurwitz partition}
\end{figure}
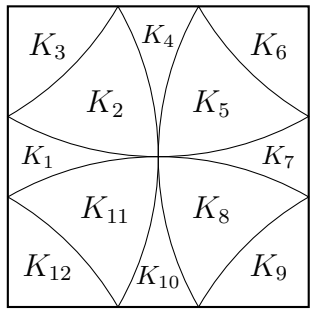

\begin{prop}\label{Hurwitz is good}
	The nearest integer algorithm satisfies the finite building property with $\Pc = \{K_1,\ldots,K_{12}\}$ from Equation~\ref{Hurwitz partition}.
\end{prop}

\begin{proof}
The nearest integer algorithm satisfies the condition of \Cref{tiling shortcut} with $Z = \Z[\i]$. Thus we must show only that each $S(K_i)$ can be written as a union of sets $a + K_j$ with $a \in \Z[\i]$.
\begin{align*}
	S(K_1) &= \left( 2 + \bigcup_{j=4}^{10} K_j \right) \cup \bigcup_{\substack{n \in \Z \\ n \ge 3}} (n+K) \\
	S(K_2) &= \left( 2+\i + \bigcup_{j=4}^{10} K_j \right) \cup \left( 1+2\i + \bigcup_{j=1}^{7} K_j \right) \\*&\qquad \cup \left( 1+\i + \bigcup_{j=1}^{12} K_j \right) \cup \bigcup_{\substack{n+m\ssi \in \Z[\ssi] \\ \min\{m,n\} \ge 2}} (n+m\i + K) \\
	S(K_3) &= \left( 1+\i + \bigcup_{j=4}^{7} K_j \right) \cup \left( 2+\i + \bigcup_{j=1,2,3,11,12} K_j \right) \\*&\qquad \cup \left( 1+2\i + \bigcup_{j=9}^{12} K_j \right) \cup \bigg( 2+2\i + K_{12} \bigg)
\end{align*}
In \Cref{fig Hurwitz double} the sets $S(K_1), S(K_2), S(K_3)$ are red, orange, and teal, respectively.
\begin{figure}[t]
    \includegraphics{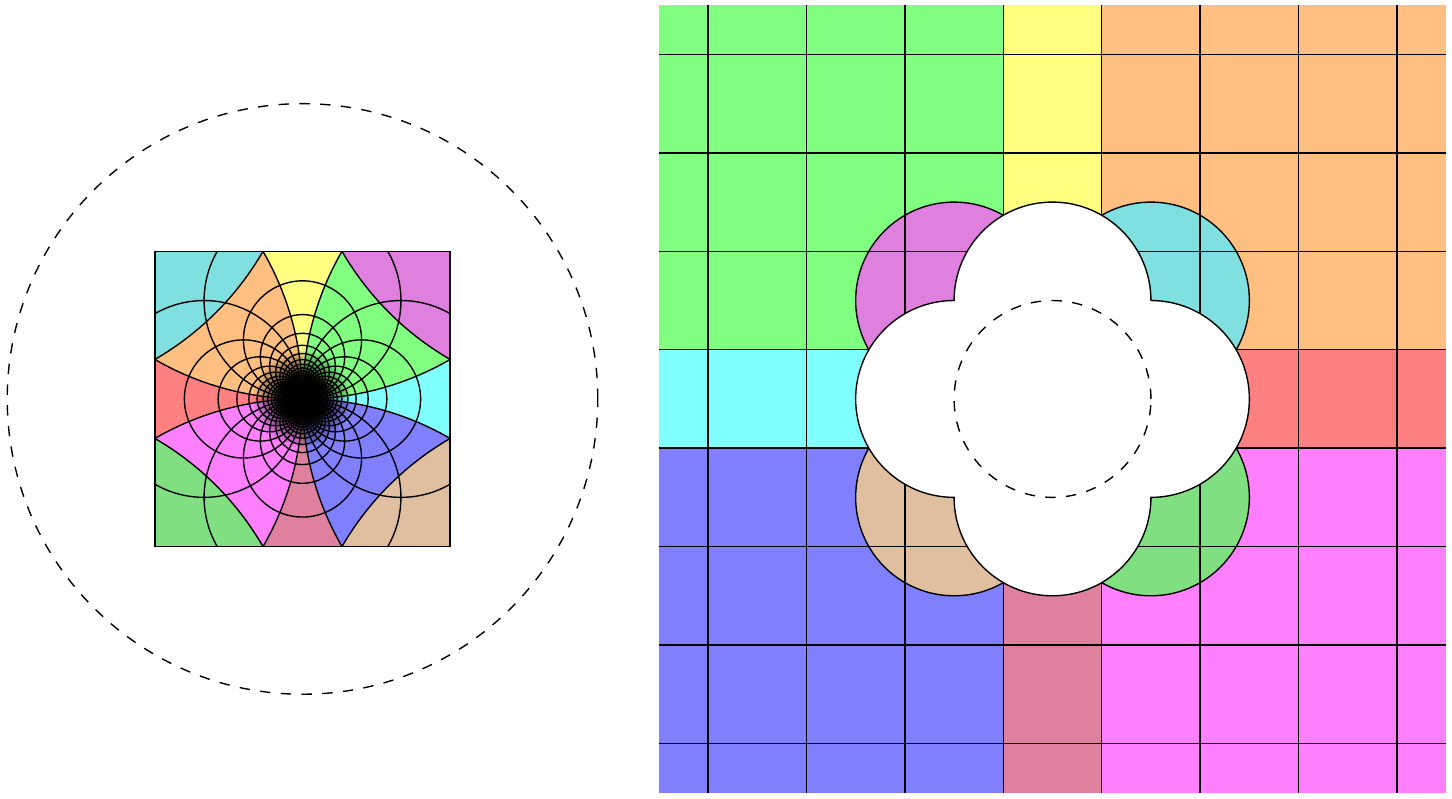}
    \MyCaption{nearest integer}
    \label{fig Hurwitz double}
\end{figure}
By symmetry, expressions for $S(K_i), 4 \le i \le 12$, will be similar.
\end{proof}

For the nearest integer algorithm, explicit expressions for $L_1,\ldots,L_{12}$ such that 
\[ \Omega = \bigcup_{i=1}^{12} K_i \times L_i \]
is a bijectivity domain of $\hat G$ are not known. Computer approximations of these sets (see \Cref{experimental}) are shown in \Cref{fig Hurwitz products} (these sets also appear in \cite[Figures 13-15]{EINN}). The sets $L_i$ appear to each have fractal boundaries, possibly a result of the fact that $\Z[\i]$-translates of $K$ perfectly tile the plane.

\begin{figure}[ht]
    \includegraphics[page=1,trim={4.75cm 19.25cm 4.75cm 2.1cm},clip,width=0.75\textwidth]{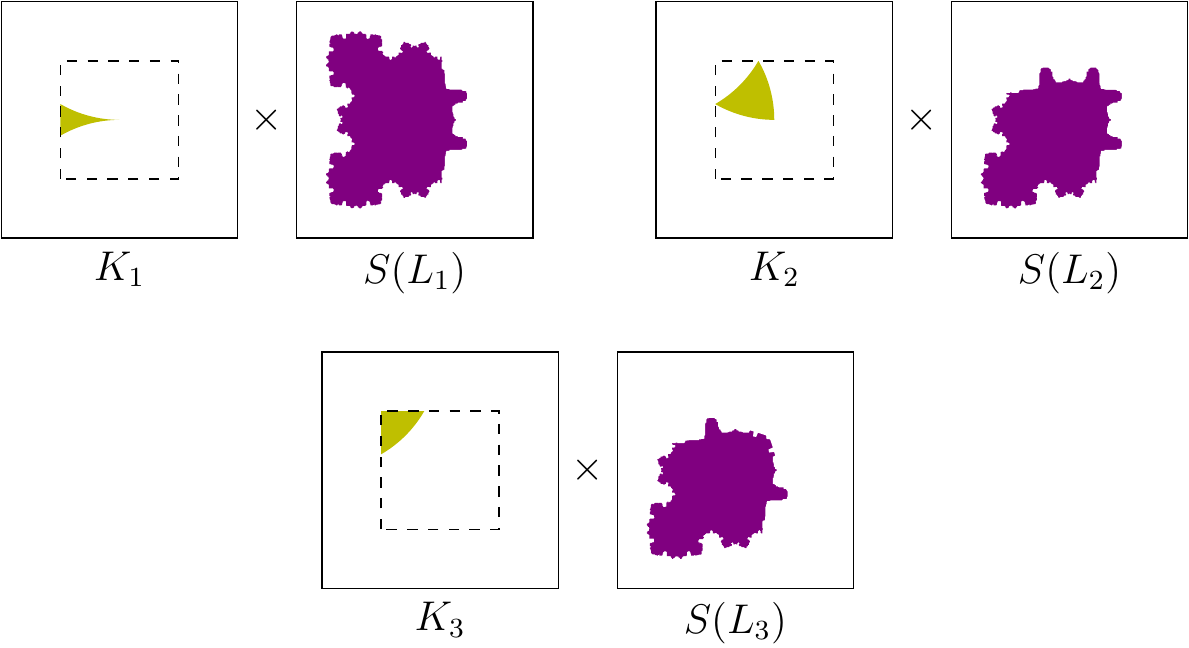}
    \caption{Approximations of some products $K_i \times S(L_i)$ for the nearest integer algorithm. The others are rotations of these}
    \label{fig Hurwitz products}
\end{figure}

\FloatBarrier
\subsection{The nearest even integer algorithm} \label{ex NE}

The \emph{nearest even algorithm} chooses the nearest Gaussian integer $x+y\i$ for which $x+y$ is even (see \Cref{even odd}). The formula
\< \label{NE choice}
	\bra{z} = \floor{\frac{\Re z + \Im z + 1}2} (1+\i) + \floor{\frac{\Re z - \Im z + 1}2} (1-\i)
\>
provides a convention to use for points equidistant from multiple even Gaussian integers.
 The fundamental set $K$ for this algorithm is a diamond with corners $\pm 1$ and $\pm \i$. This algorithm was studied by Julius Hurwitz \cite{JHurwitz} and by Tanaka \cite{T85}.

Define the following eight regions of the diamond, shown in \Cref{fig NE partition}.
\< \label{NE partition} \begin{split}
	K_1 &= \setform{ u \in K }{ \abs{z-\tfrac{-1+\ssi}2} \le \tfrac1{\sqrt2}, \abs{z-\tfrac{-1-\ssi}2} \le \tfrac1{\sqrt2} } \\
	K_2 &= \setform{ u \in K }{ \abs{z-\tfrac{-1+\ssi}2} \le \tfrac1{\sqrt2}, \abs{z-\tfrac{-1-\ssi}2} \ge \tfrac1{\sqrt2}, \abs{z-\tfrac{1+\ssi}2} \ge \tfrac1{\sqrt2} } \\
	K_i &= -\i\,K_{i-2} \quad\text{for } i = 3,\ldots,8.
\end{split} \>

\begin{figure}[ht]
    \begin{tikzpicture}
    \begin{scope}[scale=3]
    \foreach \r in {0,90,180,270} \draw [rotate=\r] (1,0) arc (-45:-225:0.707);
	\draw [thick] (1,0) -- (0,1) -- (-1,0) -- (0,-1) -- cycle;
	\foreach \k in {1,...,8} \draw (225-45*\k:0.525) node {$K_{\k}$};
	\end{scope}
    \end{tikzpicture}
    \caption{Finite partition of $K$ for the nearest even algorithm}
    \label{fig NE partition}
\end{figure}
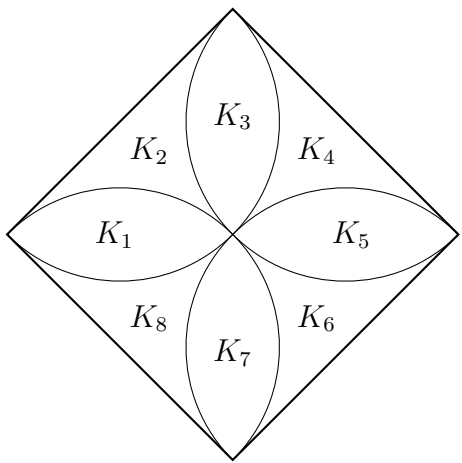

\begin{prop}\label{NE is good}
	The nearest even algorithm satisfies the finite building property with $\Pc = \{K_1,\ldots,K_8\}$ from Equation~\ref{NE partition}.
\end{prop}

\begin{proof}
For the nearest even algorithm, each $\ang a$ intersects exactly one element of $\Pc$, so we can use \Cref{Markov shortcut}. We want to express each $S(\ang a)$ in the form $a + \bigcup_{j \in J} K_j$.

\begin{figure}[t]
	\includegraphics{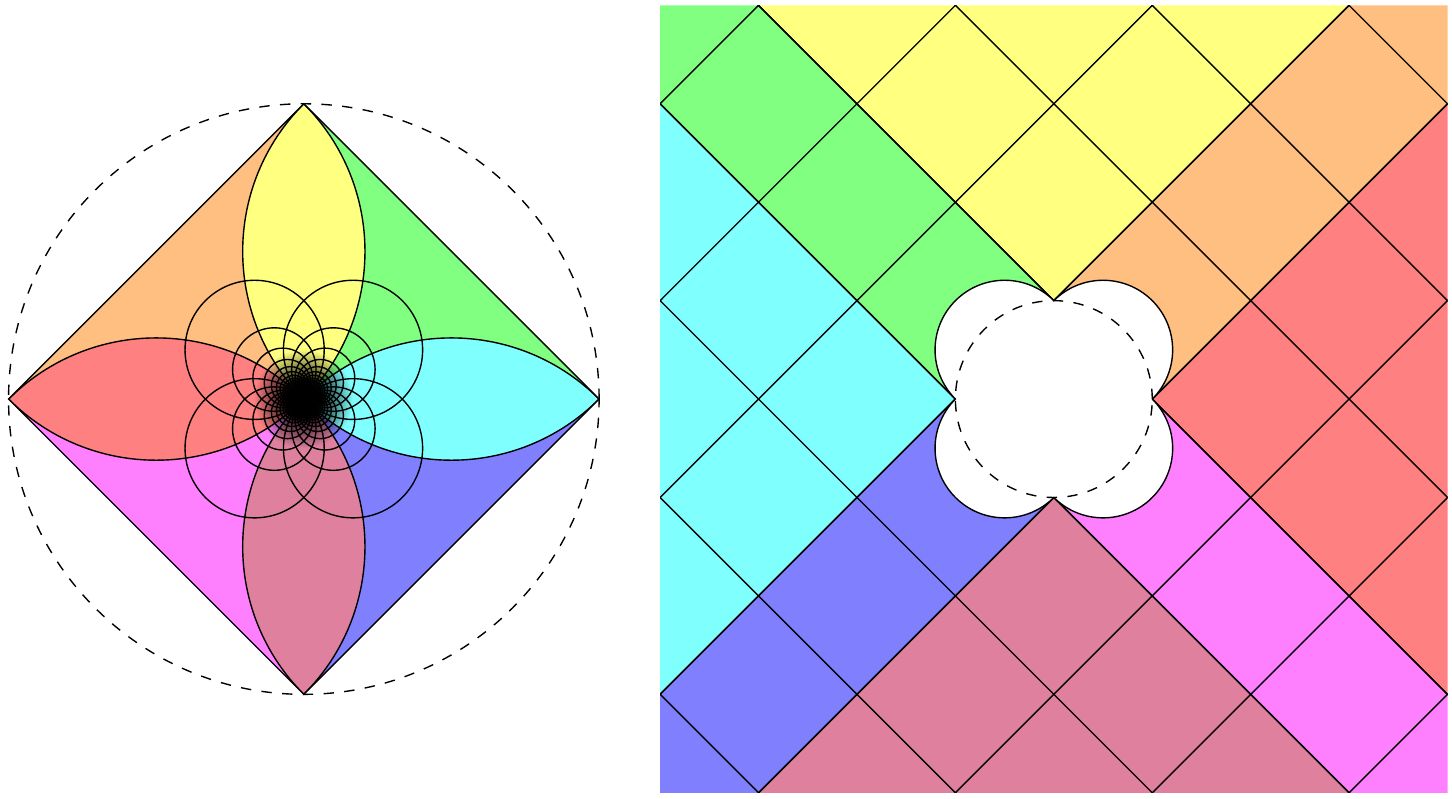}
    \MyCaption{nearest even}
    \label{fig NE double}
\end{figure}

For $a \in \C$ even with $\abs a \ge 2$, we have $S(\ang a) = a + K = a + \bigcup \Pc$.
 
\begin{samepage}
For $\abs a = \sqrt2$, we have
\begin{align*}
    S\big(\ang{1\!+\!\i}\big) &= (1\!+\!\i) + \bigcup_{j=2}^6 K_j, &
    S\big(\ang{-1\!+\!\i}\big) &= (-1\!+\!\i) + \bigcup_{j=1}^4 K_j \cup K_8, \\
    S\big(\ang{-1\!-\!\i}\big) &= (-1\!-\!\i) + \!\!\!\!\!\!\bigcup_{j \in \{1,2,6,7,8\}}\!\!\!\!\!\! K_j, &
    S\big(\ang{1\!-\!\i}\big) &= (1\!-\!\i) + \bigcup_{j=4}^8 K_j.
\end{align*}
Since $\ang 0 = \emptyset$ and $\ang a = \emptyset$ for odd $a$, this covers all $a \in \Z[\i]$.
\end{samepage}
\end{proof}

\begin{remark}
	Tanaka uses this same partition $\{K_1,\ldots,K_8\}$ (with different indices) in \cite{T85} and does express each $G(\ang a)$ as a union of elements from the partition. The two algorithms in \cite{T85} both satisfy the assumptions of \Cref{Markov shortcut}, which nearest integer does not, and this is why for the nearest integer algorithm we must use $K_{i,a} = K_i \cap \ang a$ instead of just $\ang a$.
\end{remark}

\begin{figure}[ht]
    \includegraphics[page=2,trim={4.65cm 22.75cm 4.8cm 2.1cm},clip,width=0.75\textwidth]{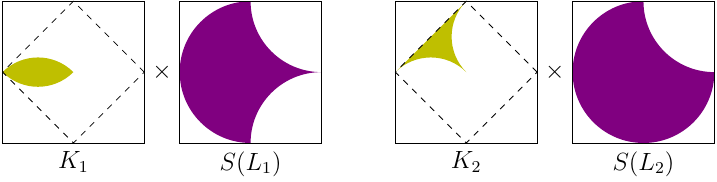}
    \caption{Some products $K_i \times S(L_i)$ for the nearest even algorithm. The others are rotations of these} 
    \label{fig NE products}
\end{figure}

For the nearest integer algorithm only computer-generated approximations of $L_i$ were shown (\Cref{fig Hurwitz products}) but for the nearest even algorithm we can explicitly describe each $L_i$---see \eqref{NE L} and \Cref{fig NE products}---and prove that \eqref{system A} holds.

\begin{samepage}
\begin{thm} \label{NE bijectivity}
	Let $K_1,\ldots,K_8$ be as in \eqref{NE partition}, and define 
	\< \label{NE L} \begin{split}
		L_1 &= \bar\C \setminus \big( B(0) \cup B(-1+\i) \cup B(-1-\i) \big) \\
		L_2 &= \bar\C \setminus \big( B(0) \cup B(-1+\i) \big) \\
		L_i &= -\i \, L_{i-2} \quad\text{for } i=3,\ldots,8. \\[-0.5em]
	\end{split} \>
	The map $\hat G$ for the nearest even algorithm is bijective a.e. on $\bigcup_{i=1}^8 K_i \times L_i$.
\end{thm}
\end{samepage}

\begin{proof}
	By \Cref{fps iff system A}, this is equivalent to showing that 
	\< \label{NE union examples} L_i = \bigcup_{(a,j) \in \ExtAlpha_i} T^{-a}SL_j \qquad\text{for } i=1,\ldots,8, \>
	where $\ExtAlpha_i = \setform{ (a,j) }{ T^a K_i \subset S(K_{j,a}) }$ (see \eqref{ExtAlpha}).
	We show proofs here for $i=1$ and $i=2$; the other cases are by symmetry.
	Since these involve $S(L_j)$, we compute
	\< \label{NE SL} \begin{split}
		SL_1 &= \bar\D \setminus \big( B(1+\i) \cup B(1-\i) \big) \\
		SL_2 &= \bar\D \setminus B(1+\i) \\
		SL_j &= \i \, SL_{j-2} \quad\text{for } j=3,\ldots,8.
	\end{split} \>
	Note $SL_1$ and $SL_2$ are shown in purple in \Cref{fig NE products}. \Cref{fig NE unions} shows $L_1$ and $L_2$ as unions of the form $\bigcup T^{-a}SL_j$. To prove \eqref{NE union examples}, it remains only to show that the $(a,j)$ pairs in these unions are indeed $(a,j) \in \ExtAlpha_i$.
	
	\begin{figure}[ht]
    \includegraphics[scale=0.85]{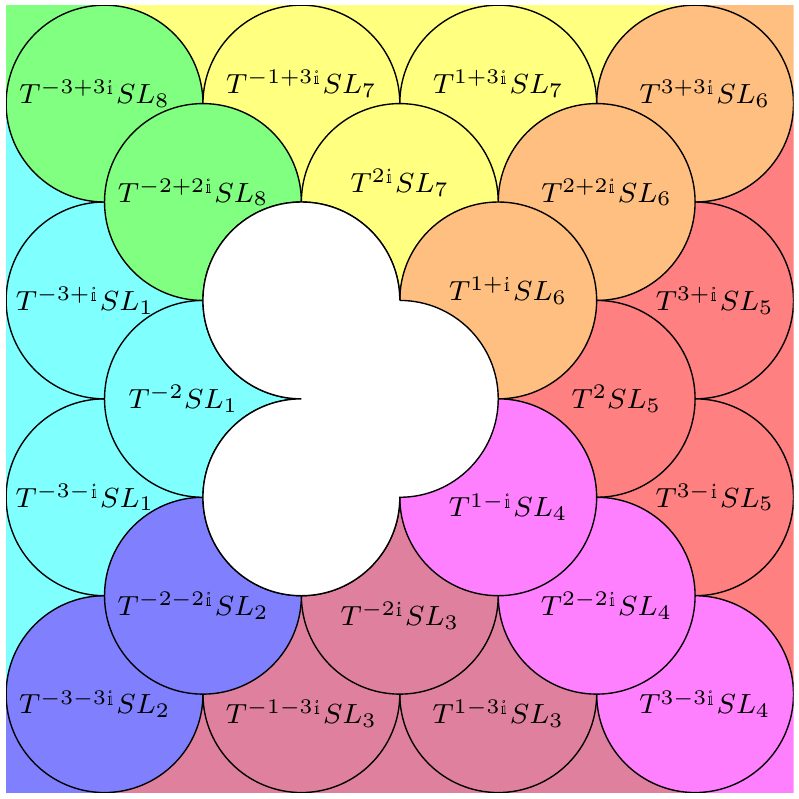} \quad \includegraphics[scale=0.85]{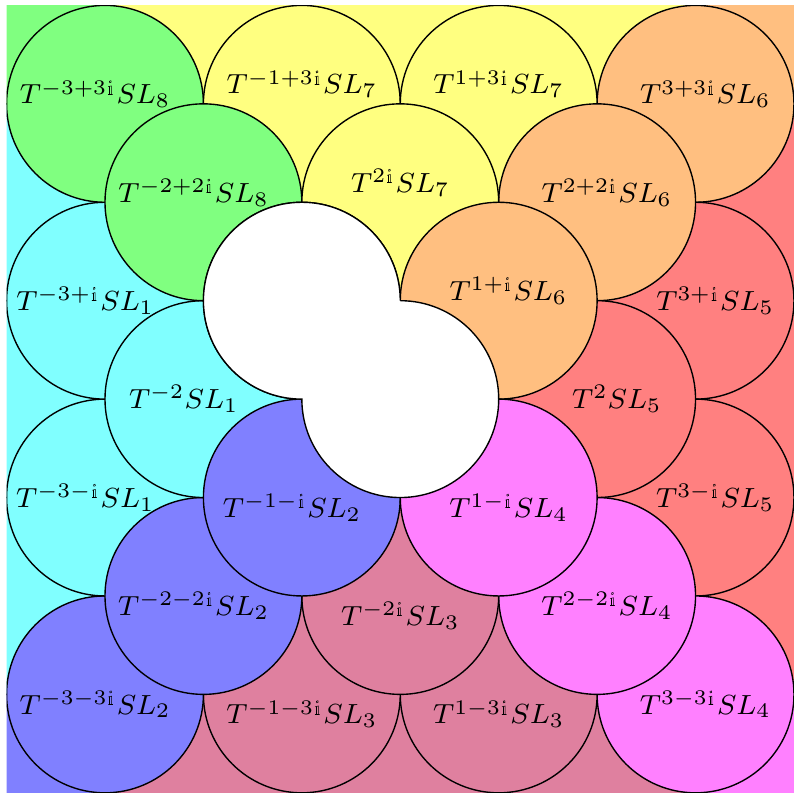}
    \caption{$L_1$ (left) and $L_2$ (right) as unions of sets $T^{-a}SL_j$, colored by $j$}
    \label{fig NE unions}
	\end{figure}

	Which $(a,j)$ satisfy $T^a K_i \subset S K_{j,a}$ for a fixed $i$?
	For the nearest even algorithm each $\ang a$ is contained in exactly one $K_j$, so each $K_{j,a}$ is either empty or is exactly~$\ang a$. Because $T^a K_i$ is never an empty set, requiring $T^a K_i \subset S K_{j,a}$ already rules out empty $K_{j,a}$. Thus we really need only $T^a K_i \subset S(\ang a)$.
	As shown in the proof of \Cref{NE is good}, we have
	\[ S\ang a = a + \bigcup_{k \in J(a)} K_k, \]
	where the sets $J(a)$ are given by
	\< \label{NE Js} \begin{split}
		J(0) &= \emptyset \\
    	J(1+\i)  &= \{ 2,3,4,5,6 \} \\
		J(-1+\i) &= \{ 1,2,3,4,8 \} \\
		J(-1-\i) &= \{ 1,2,6,7,8 \} \\
		J(1-\i)  &= \{ 4,5,6,7,8 \} \\
		J(a) &= \{1,\ldots,8\} \text{ if } \abs a \ge 2.
	\end{split} \>
	Our requirement on $(a,j)$ is really $a + K_i \subset a + \bigcup_{k \in J(a)} K_k$, which is exactly equivalent to $i \in J(a)$. That is,
	\[ (a,j) \in \ExtAlpha_i \qquad\iff\qquad i \in J(a) \text{ and } j = j(a) \] 
	where $j(a) \in \{1,\ldots,N\}$ is such that $\ang a \subset K_{j(a)}$.
	The system \eqref{NE union examples} can thus be written as
	\< \label{NE union examples 2} L_i = \bigcup_{\substack{a\in\Z[\ssi] \\ J(a) \ni i}} T^{-a}S L_{j(a)} \qquad\text{for } i=1,\ldots,8. \>
	
	For $i = 1$, we look at which $a \in \Z[\i]$ satisfy $1 \in J(a)$. From \eqref{NE Js}, this is all even $a$ except for $a = 0, 1+\i, 1-\i$. Indeed, in the left of \Cref{fig NE unions} we see $\bigcup_{a,j} T^{-a}SL_j$ for exactly those $a \ne 0, 1+\i, 1-\i$ (since we use $T^{-a}z = z - a$, the blanks are around $0$, $-1-\i$, and $-1+\i$). Examination of the $j$'s in this union match exactly $j = j(a)$ as well---note that the colors of the partial disks in \Cref{fig NE unions} exactly match the colors of the full/partial diamonds on the right of \Cref{fig NE double}.
	
	For $i = 2$, we look at when $2 \in J(a)$. From \eqref{NE Js}, this is when $a \ne 0, 1-\i$. The blanks on the right of \Cref{fig NE unions} are exactly around $0$ and $-(1-\i) = -1+\i$, and the coloring by $j$ again shows that $j=j(a)$ for each partial disk.
	
	Having shown \eqref{NE union examples 2} for $i=1,2$ and using symmetry for $i=3,\ldots,8$, then \Cref{fps iff system A} gives the result.
\end{proof}

\FloatBarrier
\subsection{The nearest odd algorithm} \label{ex NO}

The \emph{nearest odd algorithm} chooses the nearest Gaussian integer $x+y\i$ for which $x+y$ is odd.

\begin{remark} \label{rem same K}
The nearest even and nearest odd algorithms, as well as the diamond algorithm in \Cref{ex diamond}, have the same fundamental set
\[ K = \setform{ z \in \C }{ \abs{\Re z} + \abs{\Im z} \le 1 } \]
but are distinct algorithms. Compare parts (b), (c), and (d) of~\Cref{fig regions}.
\end{remark}

The finite partition for the nearest odd algorithm contains the following $12$ sets, shown in \Cref{fig NO partition}.
\< \label{NO partition} \begin{split}
	K_1 &= \setform{ u \in K }{ \abs{z-\tfrac{-1-\ssi}2} \le \tfrac1{\sqrt2}, \abs{z-\tfrac{-1+\ssi}2} \le \tfrac1{\sqrt2} } \\
	K_2 &= \setform{ u \in K }{ \abs{z-\tfrac{-1-\ssi}2} \ge \tfrac1{\sqrt2}, 0 \le \Im z \le -\Re z } \\
	K_3 &= \setform{ u \in K }{ \abs{z-\tfrac{1+\ssi}2} \ge \tfrac1{\sqrt2}, -\Im z \le \Re z \le 0 } \\
	K_i &= -\i\,K_{i-3} \quad\text{for } i = 4,\ldots,12.
\end{split} \>

\begin{figure}[ht]
    \begin{tikzpicture}
    \begin{scope}[scale=3]
    \foreach \r in {0,90,180,270} { \draw [rotate=\r] (1,0) arc (-45:-225:0.707); \draw [rotate=\r] (0,0) -- (1/2,1/2); }
	\draw [thick] (1,0) -- (0,1) -- (-1,0) -- (0,-1) -- cycle;
	\foreach \k in {1,4,7,10} \draw (210-30*\k:0.55) node {$K_{\k}$};
	\foreach \k in {2,3,5,6,8,9,11,12} \draw (210-30*\k:0.6) node {\footnotesize$K_{\k}$};
	\end{scope}
    \end{tikzpicture}
    \caption{Finite partition of $K$ for the nearest odd and diamond algorithms}
    \label{fig NO partition}
\end{figure}
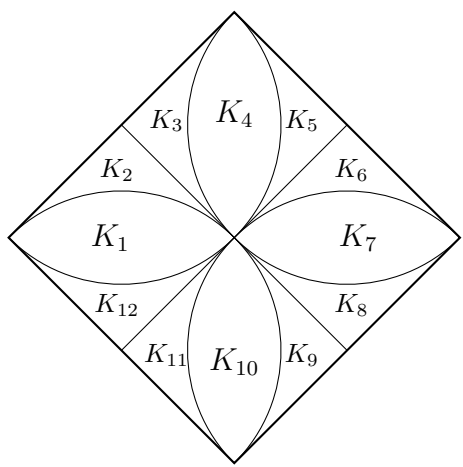

\begin{figure}[hbt]
    \includegraphics{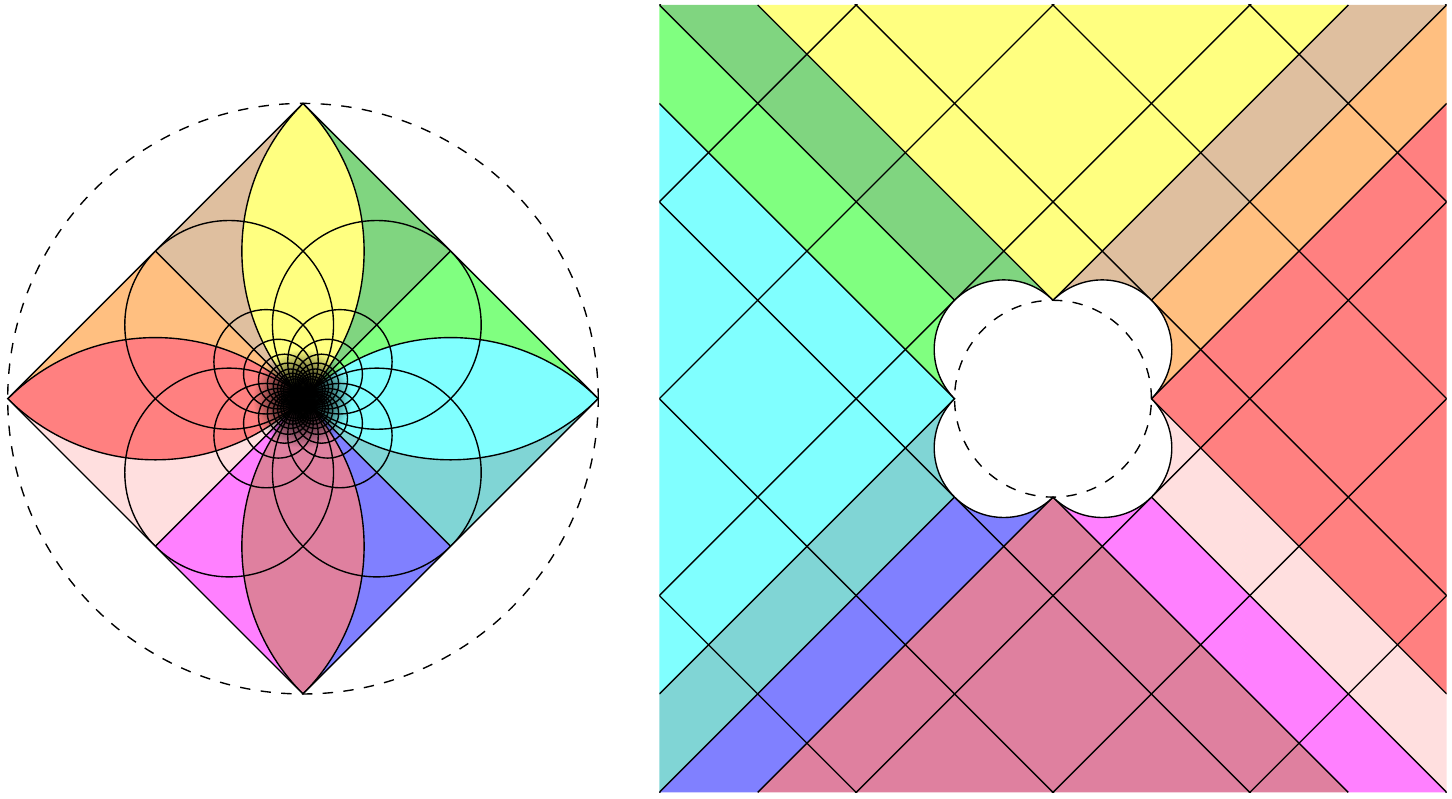}
    \MyCaption{nearest odd}
    \label{fig NO double}
\end{figure}
 
\begin{prop}\label{NO is good}
	The nearest odd algorithm satisfies the finite building property with $\Pc = \{K_1,\ldots,K_{12}\}$ from Equation~\ref{NO partition}.
\end{prop}

\begin{proof}
	Unlike nearest even, the nearest odd algorithm does not satisfy the assumptions of \Cref{Markov shortcut}. It does, however, satisfy the conditions of \Cref{tiling shortcut} with 
	\[ Z = \setform{ x+y\i \in \Z[\i] }{ x+y\text{ is odd} }. \]
    We want to express each $S(K_i)$ as a union of sets $a + K_j$.
    For $i=1$ we have
    \begin{align*}
		S(K_1) &= \left(1+\bigcup_{j=6}^8 K_j\right) \cup
		\bigcup_{n\in\N} \left(n+1+n\i + \bigcup_{j=6}^{11}K_j\right) 
		\\*&\qquad \cup \bigcup_{n\in\N} \left(n+1-n\i + \bigcup_{j=3}^{8}K_j\right) \cup
		\bigcup_{a \in A} \bigg(a + K\bigg),
	\end{align*}
	where $\N = \{1,2,3,\ldots\}$ and $A$ is the set of $m+n\i \in \Z[\i]$ such that $m+n$ is odd, $n \ge 2$, and $-m+2 \le n \le m-2$.
	
	For $i=2,3$, we have 
	\begin{align*}
		S(K_2) &= \bigg(1+K_5\bigg) \cup \bigcup_{n\in\N} \left((n+1)+n\i + \bigcup_{j=1}^{5}K_j \cup K_{12}\right) \\
		S(K_3) &= \bigg(\i+K_6\bigg) \cup \bigcup_{n\in\N} \left(n+(n+1)\i + \bigcup_{j=6}^{11}K_j\right)\!.
	\end{align*} 
	The other $S(K_i)$ are similar, and the proof is complete by \Cref{tiling shortcut}.
\end{proof}

\FloatBarrier
\begin{samepage}
\subsection{The diamond algorithm} \label{ex diamond}

The \emph{diamond algorithm}, discussed in \cite{myMasters}, uses the choice function constructed as in \Cref{make choice} with $X$ being the diamond with corners $\pm 1$ and $\pm \i$. As mentioned in \Cref{rem same K}, its fundamental set 
\[ K = \setform{ z }{ \abs{\Re z} + \abs{\Im z} \le 1 } \]
is also the fundamental set for the nearest even and nearest odd algorithms. The finite partition of $K$ is also the same partition used with the nearest odd algorithm---see \eqref{NO partition} and \Cref{fig NO partition}.
\end{samepage}

\begin{prop}\label{diamond is good}
	The diamond algorithm satisfies the finite building property with $\Pc = \{K_1,\ldots,K_{12}\}$ from Equation~\ref{NO partition}.
\end{prop}

\begin{proof}
	This is the only algorithm discussed in this paper for which neither \Cref{tiling shortcut} nor \Cref{Markov shortcut} apply (see \Cref{tab shortcuts}), so we will use \Cref{pre tiling shortcut}.
	
    \begin{figure}[hbt]
        \includegraphics{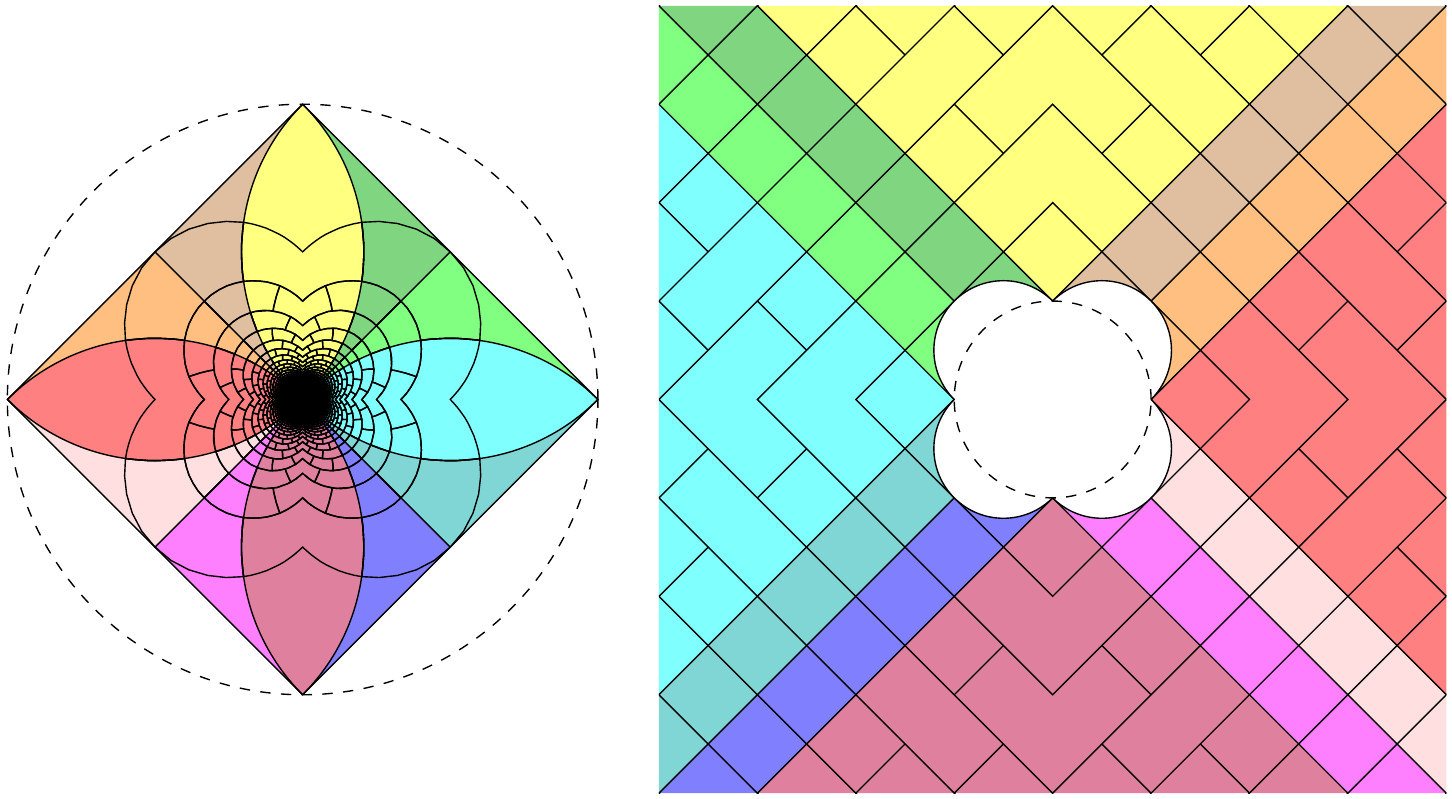}
        \MyCaption{diamond}
        \label{fig diamond double}
    \end{figure}
    
	From \Cref{fig diamond double}, one can see that each $S(K_i)$ is a union of sets of the form $a + K_j$. However, care must be taken in when determining $\bra{a+K_j}$. For example,
	\begin{align*}
		\bra{ (2+2\i) + K_8 } &= 2+2\i, \\*
		\bra{ (2+2\i) + K_9 } &= 2+\i \qquad\text{(not $2+2\i$)}
	\end{align*}
	because
	\< \label{diamond cell QI} S(\ang a) = a + \bigcup_{j=3}^{8} K_j \quad\text{if }\Re a\ge1\text{ and }\Im a\ge1. \>
	
	To use \Cref{pre tiling shortcut}, we must show not only that each $S(K_i)$ is a union of sets $a + K_j$ but also that this can be done with $\bra{a+K_j} = a$ for each set in the union. 
	
	The symmetry group of the diamond $K$ is $\mathrm{Dih}_4$ (order $8$), and indeed for any $\xi : \C \to \C$ in $\mathrm{Dih}_4$ (here $\C \cong \R^2$) we have that 
	\[ \bra{\xi z} = \xi \bra{z}. \]
	This implies that 
	\[ S(\ang{\xi a}) = \xi S(\ang a) \]
	for all $\xi \in \mathrm{Dih}_4$ and all $a \in \Z[\i]$. The partition in \eqref{NO partition} also respects this symmetry: for each $\xi \in \mathrm{Dih}_4$ there exists a permutation $\sigma_{\xi} : \{1,\ldots,12\} \to \{1,\ldots,12\}$ such that $\xi(K_i) = K_{\sigma_{\xi}(i)}$ for all $1 \le i \le 12$. Therefore we need only to show that the conditions of \Cref{pre tiling shortcut} are satisfied for $K_1$ and $K_2$. After that, symmetry will handle all remaining $K_i$.
	
	\medskip
	First, consider $K_2$.
	\[ S(K_2) = \setform{ w \in \C }{ \abs{w-(\tfrac12 + \tfrac12\i)} \ge 1, \Re w \ge 1, \Re w - 1 \le \Im w \le \Re w }, \]
	which is orange on the right of \Cref{fig diamond double},
	may be decomposed into
	\[ S(K_2) = C \cup \bigcup_{n \in \N} A_n \cup \bigcup_{n \in \N, n \ge 2} B_n, \]
	where $C$ is the curved set $1+K_5 = 1+\i+K_9$ and $A_n$ and $B_n$ are the countably many small (side length $\frac1{\sqrt2}$) diamonds 
	\[ A_n := \big(n+n\i\big) + \big(K_6 \cup K_7 \cup K_8\big) = \big((n+1)+n\i\big) + \big(K_{12} \cup K_1 \cup K_2\big). \]
	and 
	\[ B_n := \big(n+n\i\big) + \big(K_9 \cup K_{10} \cup K_{11}\big) = \big(n+(n-1)\i\big) + \big(K_3 \cup K_4 \cup K_5\big). \]
	Using \eqref{diamond cell QI}, we see that $\bra{A_n} = n+n\i$ for $n \ge 1$ and $\bra{B_n} = n+(n-1)\i$ for $n \ge 2$. Additionally, $S(\ang{1}) = 1 + \bigcup_{j=5}^9 K_j$ gives that $\bra{C} = 1$. Therefore 
	\[ S(K_2) = \bigg(1+K_5\bigg) \bigcup_{n \in \N} \bigg(n+n\i + K_6 \cup K_7 \cup K_8\bigg) \cup \bigcup_{\substack{n\in\N\\ n\ge2}} \bigg(n+(n-1)\i + K_3 \cup K_4 \cup K_5\bigg) \]
	is a union of the form required by \Cref{pre tiling shortcut}.
	
	\medskip
	Now consider $K_1$.
	\[ S(K_1) = \setform{ w \in \C }{ -\tfrac\pi4 \le \arg(z-1) \le \tfrac\pi4 } \]
	is red on the right of \Cref{fig diamond double}. For this we will immediately write the correct union expression and then verify it:
	\< \label{diamond SK1 union}
		S(K_1) = 
		\bigcup_{\substack{n\in\N\\ n\ge2}} A_n
		\cup \hspace{-1em}\bigcup_{\substack{a\in\Z[\ssi]\\ 0<\arg a<\pi/4}}\hspace{-1em} B_a
		\cup \;C\;
		\cup \bigcup_{\substack{n\in\N\\ n\ge2}} D_n
		\cup \hspace{-1em}\bigcup_{\substack{a\in\Z[\ssi]\\ -\pi/4<\arg a<0}}\hspace{-1em} E_a
		\cup \bigcup_{\substack{n\in\N\\ n\ge2}} F_n,
	\>
	where
	\begin{align*}
		A_n &=\textstyle n+(n-1)\i + \bigcup_{j=6}^{8} K_j \\
		B_a &=\textstyle a + \bigcup_{j=3}^{8} K_j \\
		C &=\textstyle 1 + \bigcup_{j=6}^8 K_j \\
		D_n &=\textstyle n + \bigcup_{j=3}^{11} K_j \\
		E_a &=\textstyle a + \bigcup_{j=6}^{11} K_j \\
		F_n &=\textstyle n-(n-1)\i + \bigcup_{j=6}^{8} K_j.
	\end{align*}
	We use \eqref{diamond cell QI} to get $\bra{A_n} = n+(n-1)\i$ and $\bra{B_a} = a$. As before, $S(\ang{1}) = 1 + \bigcup_{j=5}^9 K_j$ gives that $\bra{C} = 1$ (this is a different $C$ than was used for the discussion of $K_1$, but $\bra{C} = 1$ for both). To set $\bra{E_a} = a$ and $\bra{F_n} = n-(n-1)\i$, we use a symmetric version of \eqref{diamond cell QI}, namely,
	\[ S(\ang a) = a + \bigcup_{j=6}^{11} K_j \quad\text{if }\Re a\ge1\text{ and }\Im a\le-1. \]
	Altogether, we now have that $\bra{a+K_j} = a$ for every term in the union \eqref{diamond SK1 union}. \Cref{pre tiling shortcut} then proves that the diamond algorithm satisfies the finite building property.
\end{proof}

\begin{thm}
	Let $K_1,\ldots,K_{12}$ be as in \eqref{NO partition}, and define 
	\< \label{diamond L} \begin{split}
		L_1 &= \bar\C \setminus \big( B(\i) \cup B(0) \cup B(-\i) \cup \{ w:\Re w < \tfrac{-1}2 \}\big) \\
		L_2 &= \bar\C \setminus \big( B(\i) \cup B(0) \cup \{ w:\Re w < \tfrac{-1}2 \}\big) \\
		L_3 &= \bar\C \setminus \big( B(1) \cup B(0) \cup \{ w:\Im w < \tfrac{-1}2 \}\big) \\
		L_i &= -\i L_{i-3} \quad\text{for } i=4,\ldots,12.
	\end{split} \>
	The map $\hat G$ for the diamond algorithm is bijective a.e. on $\bigcup_{i=1}^{12} K_i \times L_i$.
\end{thm}

This is essentially \cite[Theorem 8]{myMasters}, although in that paper a ``slow'' map is used, analogous to the real map $h:\R \to \R$
	\[ h(x) = \left\{ \begin{array}{ll} x+1 &\text{if } x \le a \\ -1/x &\text{if } a\le x< b \\ x-1 &\text{if } x \ge b \end{array} \right. \]
	as opposed to the Gauss map $g:[a,b) \to [a,b)$ given by
	\[ g(x) = \frac{-1}x - \bra{\frac{-1}x} \]
	(``slow'' because for each $x$ we have $g(x) = h^n(x)$ some some $n$). \\

\begin{figure}[ht]
    \includegraphics[page=4,trim={4.65cm 22.75cm 4.8cm 2.1cm},clip,width=0.75\textwidth]{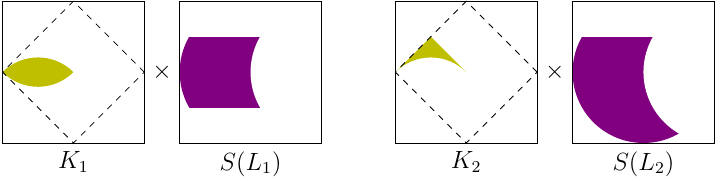}
    \caption{Some products $K_i \times S(L_i)$ for the diamond algorithm. The others are rotations or reflections of these}
    \label{fig diamond products}
\end{figure}

\FloatBarrier
\subsection{The disk algorithm} \label{ex disk}

In \cite{T85}, Tanaka discuses two complex plus continued fraction algorithms that both use only digits in the ideal generated by $\alpha := 1 + \i$ in the ring $\Z[\i]$, that is, the set
\[
	E = (\alpha) = \setform{ n \alpha + m \overline\alpha }{ n,m \in \Z }.
\]
It is worth pointing out that this set can be equivalently defined as
\< \label{E}
	E = \setform{ x+y\i \in \Z[\i] }{ x+y \text{ is even} }.
\>
The first of Tanaka's algorithms is the nearest even algorithm described previously. The second, called here the \emph{disk algorithm}, is described as follows.
Define nine subsets (highly overlapping) of the unit disk $\bar\D$ by
\< \label{Vs}
	\begin{split}
		V_0 &= \bar\D
		\qquad
		V_1 = \setform{ w \in \bar\D }{ \abs{w+\alpha} \ge 1 }
		\qquad
		V_5 = V_1 \cap V_2 \\
		V_j &= -\i \, V_{j-1} \quad\text{for } j=2,3,4,6,7,8
	\end{split}
\>
These are shown in Figure~\ref{fig Vs}.
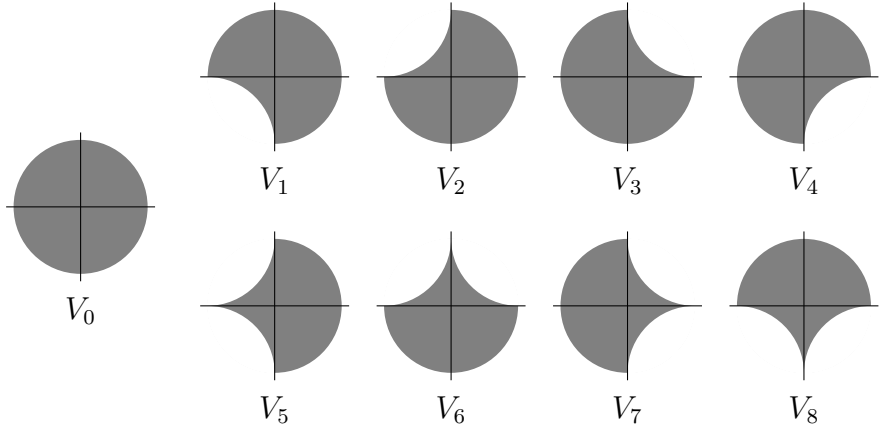
\begin{figure}[h]
    \newcommand\tikzV[1]{\begin{tikzpicture}[scale=0.9] \fill [gray] (0,0) circle (0.99); \begin{scope} \clip (0,0) circle (1); #1 \end{scope} \draw (-1.1,0) -- (1.1,0); \draw (0,-1.1) -- (0,1.1); \end{tikzpicture}}
    $ \begin{array}{c} \tikzV{} \\ V_0 \end{array}\;\;
    	\begin{array}{cccc}
    	\tikzV{ \fill [white] (-1,-1) circle (1); } &
    	\tikzV{ \fill [white] (-1,1) circle (1); } &
    	\tikzV{ \fill [white] (1,1) circle (1); } &
    	\tikzV{ \fill [white] (1,-1) circle (1); } \\
    	V_1 & V_2 & V_3 & V_4 \\[1em]
    	\tikzV{ \fill [white] (-1,-1) circle (1) (-1,1) circle (1); } &
    	\tikzV{ \fill [white] (-1,1) circle (1) (1,1) circle (1); } &
    	\tikzV{ \fill [white] (1,1) circle (1) (1,-1) circle (1); } &
    	\tikzV{ \fill [white] (1,-1) circle (1) (-1,-1) circle (1); } \\
    	V_5 & V_6 & V_7 & V_8 \\[1em]
    \end{array} $
    \caption{The sets $V_0, V_1, \ldots, V_8 \subseteq \D$.}
    \label{fig Vs}
\end{figure}

\noindent The even integers $E$ are then partitioned into nine regions:
\[ \label{Es}
	\begin{split}
		E_0 &= \{0\}
		\qquad
		E_1 = \setform[\!:\!]{ n \alpha }{ n > 0 }
		\qquad
		E_5 = \setform[\!:\!]{ n\alpha + m\overline\alpha }{ n,m > 0 } \\
		E_j &= -\i \, E_{j-1} \quad\text{for } j=2,3,4,6,7,8.
	\end{split}
\]
Lastly, we denote by $V(a)$ whichever set $V_j$ satisfies $a \in E_j$. The choice function is then defined as
\< \label{disk choice}
	\bra{w}
	= a \in E \quad\text{if}\quad w \in a + V(a).
\>
In \Cref{fig regions}(e), all the colored regions along the ray $\arg z = \pi/4$ are translates of $V_1$, and all the colored regions on the right are translates of $V_5$.

\medskip
For the finite building property, we partition the disk into five regions:
\< \label{disk partition} \begin{split}
		K_2 &= \setform{ z \in \bar\D }{ \abs{z-(-1+\i)} \le 1 } \\
    	K_i &= -\i\, K_{i-1} \quad\text{for } i = 3,4,5 \\
    	K_1 &= \bar\D \setminus \big( K_2 \cup K_3 \cup  K_4 \cup K_5 \big).
	\end{split} \>

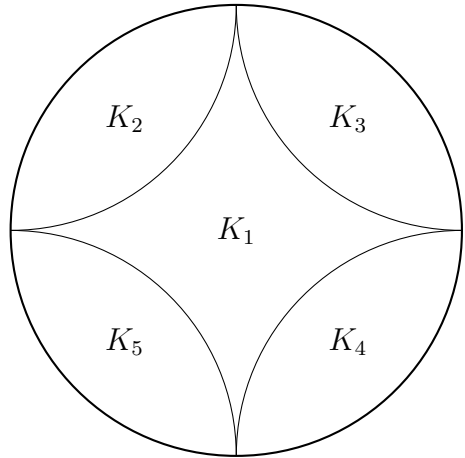
\begin{figure}[h]
    \begin{tikzpicture}
    \begin{scope}[scale=3]
    \foreach \r in {0,90,180,270} \draw [rotate=\r] (1,0) arc (-90:-180:1);
	\draw [thick] (0,0) circle (1);
	\draw (0,0) node {$K_1$};
	\foreach \k in {2,3,4,5} \draw (-45-90*\k:0.7) node {$K_{\k}$};
	\end{scope}
    \end{tikzpicture}
    \caption{Finite partition of $K$ for the disk algorithm}
    \label{fig disk partition}
\end{figure}

\begin{lem} \label{disk parts are good}
	Each set $V_j$ from \eqref{Vs} is buildable from $\{K_1,\ldots,K_5\}$.
\end{lem}

\begin{proof}
    This can be seen in \nameCref{fig disk partition}s~\ref{fig Vs} and~\ref{fig disk partition}. The explicit unions are $V_0 = \bigcup \Pc$ and
    \begin{align*}
		V_1 &= K_1 \cup K_2 \cup K_3 \cup K_4 &
		V_5 &= K_1 \cup K_3 \cup K_4 \\*
		V_2 &= K_1 \cup K_3 \cup K_4 \cup K_5 &
		V_6 &= K_1 \cup K_4 \cup K_5 \\*
		V_3 &= K_1 \cup K_2 \cup K_4 \cup K_5 &
		V_7 &= K_1 \cup K_2 \cup K_5 \\*
		V_4 &= K_1 \cup K_2 \cup K_3 \cup K_5 &
		V_8 &= K_1 \cup K_2 \cup K_3. \qedhere
    \end{align*}
\end{proof}

\begin{figure}[hbt]
    \includegraphics{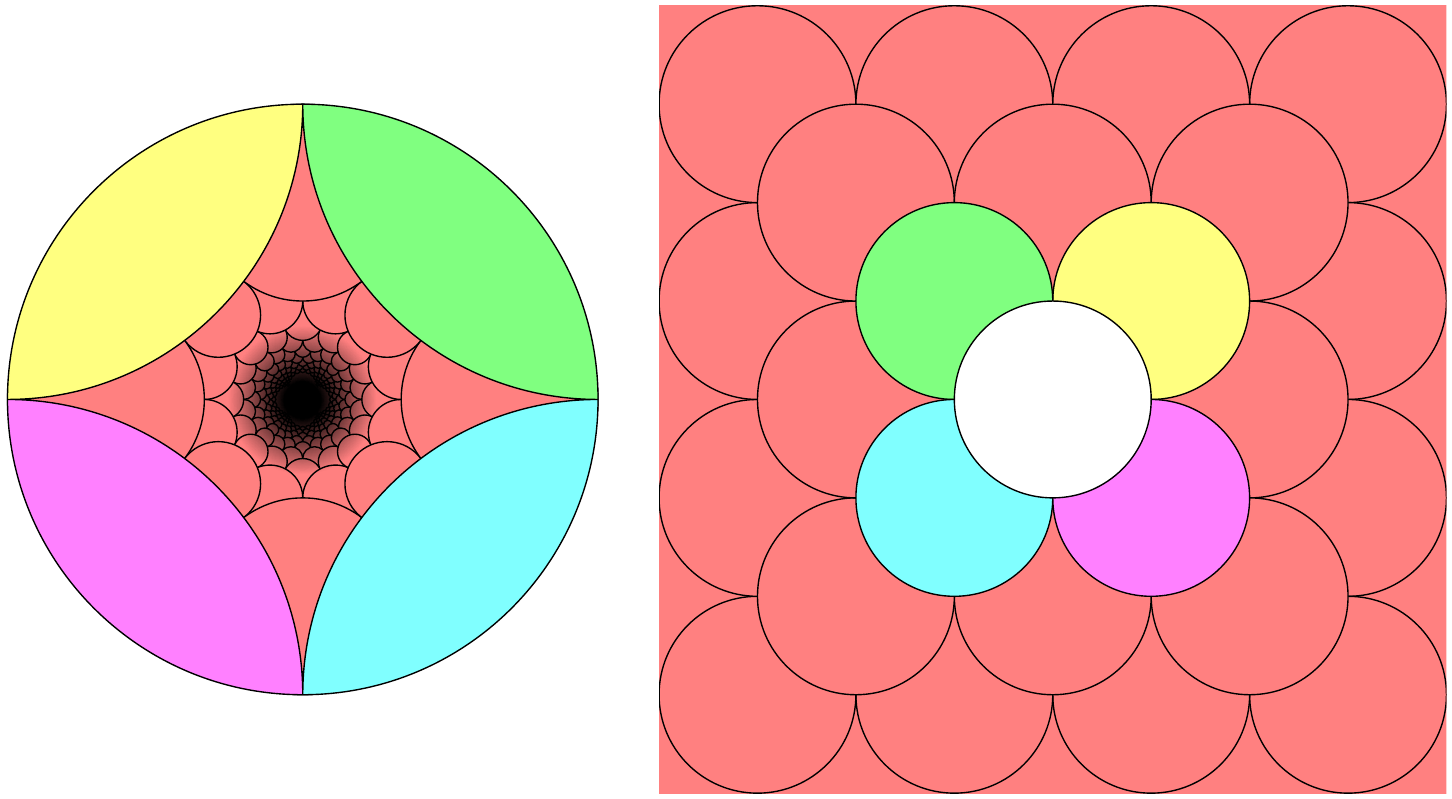}
    \MyCaption{disk}
    \label{fig disk double}
\end{figure}

\begin{prop}\label{disk is good}
	The disk algorithm satisfies the finite building property with $\Pc = \{K_1,\ldots,K_5\}$ from Equation~\ref{disk partition}.
\end{prop}

\begin{proof}
	We want to use \Cref{Markov shortcut}, which requires that each $\ang a$ is contained in some $K_i$ and that each $S(\ang a)$ is a union of sets of the form $a + K_j$.
	
	If $a \in \C$ is even and $\abs a > \sqrt 2$, then $\ang a \subset K_1$ (see \Cref{fig disk double}). For $\abs a = \sqrt2$, we have not just subsets but equality:
	\begin{align*}
		\ang{1+\i} &= K_2, &
		\ang{-1+\i} &= K_3, &
		\ang{-1-\i} &= K_4, &
		\ang{1-\i} &= K_5.
	\end{align*}
	If $a = 0$ or $a \in \C$ is not even, then $\ang a = \emptyset$. 
	Thus we have shown that each $\ang a$ is a subset of some $K_i$.
	
	By \eqref{disk choice}, 
	\[ S(\ang a) = a + V(a) \]
	and therefore $T^{-a}S\ang a = V_j$ for some $V_j$. Since each $V_j$ is buildable by \Cref{disk parts are good}, this \namecref{disk is good} is now proved via \Cref{Markov shortcut}.
\end{proof}

\begin{figure}[ht]
    \includegraphics[page=5,trim={4.65cm 22.75cm 4.8cm 2.1cm},clip,width=0.75\textwidth]{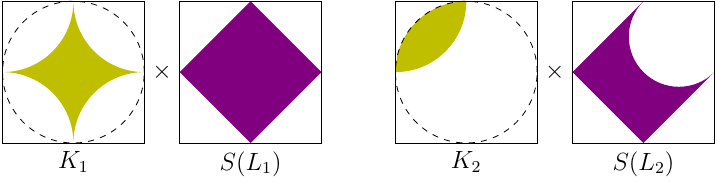}
    \caption{Some products $K_i \times S(L_i)$ for the disk algorithm. The others are rotations of $K_2 \times S(L_2)$}
    \label{fig disk products}
\end{figure}

\begin{thm} \label{disk bijectivity}
	Let $K_1,\ldots,K_5$ be as in \eqref{disk partition}, and define 
	\< \label{disk L} \begin{split}
		L_1 &= \bar\C \setminus \big( B(\tfrac{1+\ssi}2,\tfrac{1}{\sqrt2}) \cup B(\tfrac{-1+\ssi}2,\tfrac{1}{\sqrt2}) \cup B(\tfrac{-1-\ssi}2,\tfrac{1}{\sqrt2}) \cup B(\tfrac{1-\ssi}2,\tfrac{1}{\sqrt2}) \big) \\
		L_2 &= L_1 \setminus \setform{ w \in \C }{ \Im w > -\Re w + 1 } \\
		L_i &= -\i \, L_{i-1} \quad\text{for } i=3,4,5.
	\end{split} \>
	The map $\hat G$ for the disk algorithm is bijective a.e. on $\bigcup_{i=1}^5 K_i \times L_i$.
\end{thm}

The proof is similar to that of \Cref{NE bijectivity}, showing
	\[ L_i = \bigcup_{\substack{a\in\Z[\ssi] \\ J(a) \ni i}} T^{-a}S L_{j(a)} \qquad\text{for } i=1,\ldots,12, \]
	where $J(a)$ is such that $S\ang a = a + \bigcup_{k \in J(a)}K_k$ and $j(a)$ is the unique index for which $\ang a \subset K_j$.

\FloatBarrier
\subsection{The shifted Hurwitz algorithm} \label{ex SH}

In \cite[Examples~2.3\#2]{DN14}, Dani and Nogueria briefly describe the following family of choice functions indexed by $d \in \C$. For a fixed $d$, the function chooses for $z \in \C$ the point $h$ in $\overline{B(z,1)} \cap \Z[\i]$ for which $\abs{z-h-d}$ is minimal. Thus $d=0$ gives the standard Hurwitz (nearest integer) function. For $\abs d \le \frac{\sqrt{3}-1}2$ the set $K$ is a shifted square, but for larger $\abs d$ the set $K$ has curved boundary portions owing to the fact that $K$ is required to be inside $\bar\D$. The $d = -\tfrac12$ algorithm is called here the \emph{shifted Hurwitz algorithm}. See \Cref{fig regions}(f).

\begin{remark}
The shifted Hurwitz continued fraction expansions of a complex number with zero imaginary part coincides with its real $(-1,0)$-continued fraction, also called the simple backwards continued fraction.
\end{remark}

The set $K$ is partitioned into ten regions, shown in \Cref{fig SH partition}:
\< \label{SH partition} \begin{split}
	K_1 &= \setform{ z \in K }{ -\tfrac12 \le \Re z \le 0, \abs{z-\i} \ge 1, \abs{z+\i} \ge 1 } \\
	K_2 &= \setform{ z \in K }{ \abs{z-\i} \le 1, \abs{z+1} \le 1, \abs{z-(-1+\i)} \ge 1 } \\
	K_3 &= \setform{ z \in K }{ 0 \le \Im z \le \tfrac12, \abs{z-1} \ge 1, \abs{z+1} \ge 1 } \\
	K_4 &= \setform{ z \in K }{ \Re z \ge -\tfrac12, \abs{z-(-1+\i)} \le 1 } \\
	K_5 &= \setform{ z \in K }{ \Re z \le -\tfrac12, \abs{z-(-1+\i)} \le 1 } \\
	K_6 &= \setform{ z \in K }{ \Re z \le -\tfrac12, \abs{z-(-1+\i)} \ge 1, \abs{z-(-1-\i)} \ge 1 } \\
	K_i &= \setform{ \overline z }{ z \in K_{i-5} } \quad\text{for } i = 7,8,9,10.
\end{split} \>

\begin{figure}[ht]
    \begin{tikzpicture}[scale=3]
    \draw (-0.133975,1/2) arc (-30:-90:1) arc (90:30:1) arc (-30:30:1);
    \draw (-1/2,0.133975) arc (-120:-90:1) arc (90:120:1);
    \draw (-1/2,-1/2) -- (-1/2,1/2);
	\draw [thick] (0.133975,1/2) -- (0.133975-1,1/2) arc (150:210:1) -- (0.133975,-1/2) arc (210:150:1);;
	
	\foreach \k/\r in {1/180, 3/90, 8/-90} \draw (\r:0.4) node {\footnotesize$K_{\k}$};
	\draw (-0.6,0) node {\footnotesize$K_{6}$};
	\foreach \k/\r in {2/30, 4/70, 5/135, 7/-30, 9/-70, 10/-135} \draw (-1/2,0)+(\r:0.39) node {$K_{\k}$};
    \end{tikzpicture}
    \caption{Finite partition of $K$ for the shifted Hurwitz algorithm}
    \label{fig SH partition}
\end{figure}
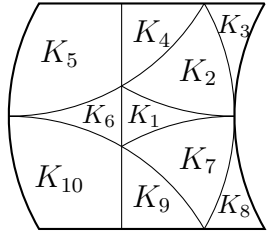

\begin{remark} For $i \notin \{5,6,10\}$, each $K_i$ in \eqref{SH partition} is also a set from the (unshifted) Hurwitz partition in \eqref{Hurwitz partition}, not necessarily with the same index. The pieces $K_5, K_6, K_{10}$ that are outside the unit square centered at the origin can also be described as unions of translates of sets from \eqref{Hurwitz partition}. \end{remark}

\begin{prop}\label{SH is good}
	The shifted Hurwitz algorithm satisfies the finite building property with $\Pc = \{K_1,\ldots,K_{10}\}$ from Equation~\ref{SH partition}.
\end{prop}

\begin{figure}[hbt]
    \includegraphics{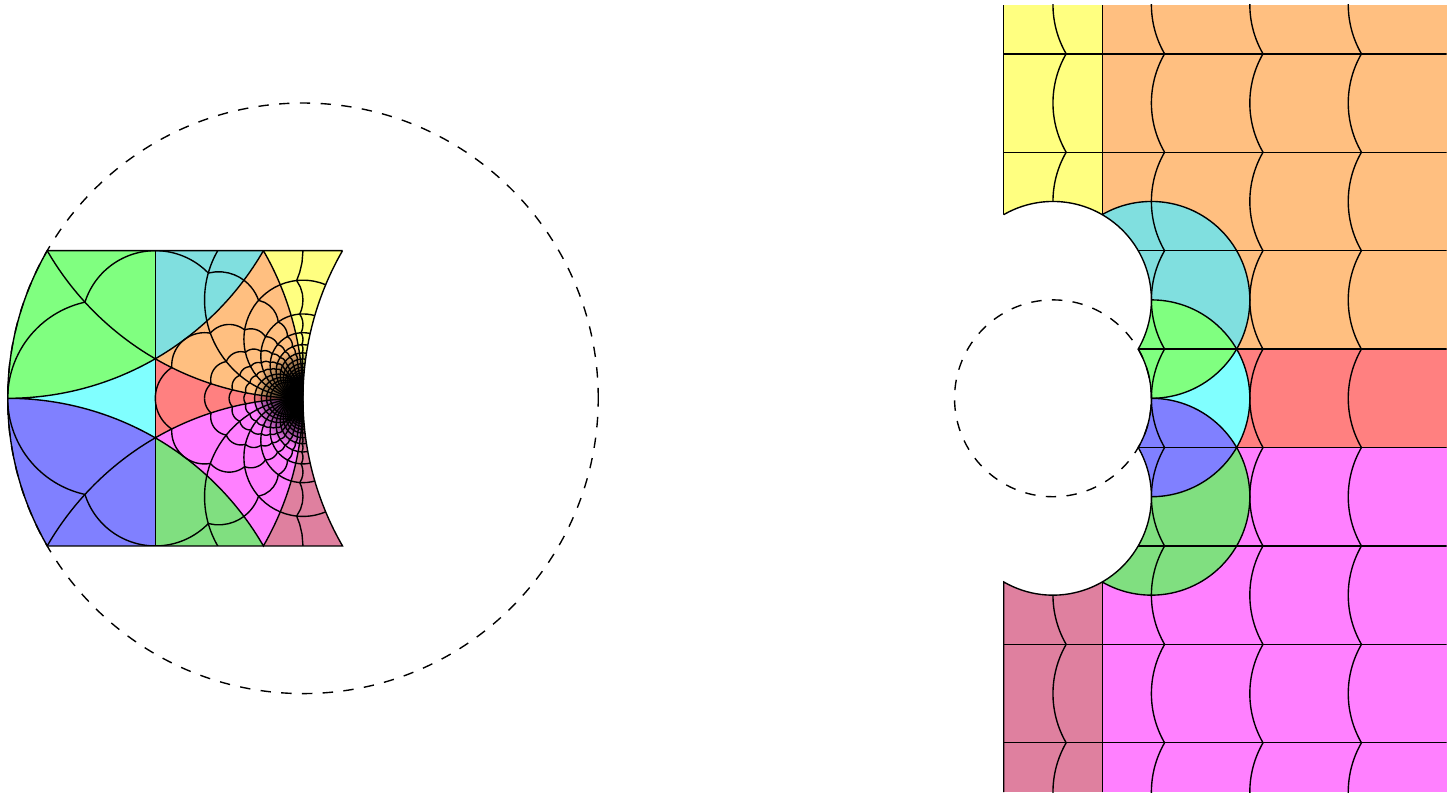}
    \MyCaption{shifted Hurwitz}
    \label{fig SH double}
\end{figure}

\begin{proof}
The shifted Hurwitz algorithm satisfies the condition of \Cref{tiling shortcut} with $Z = \Z[\i]$. Thus we must show only that each $S(K_i)$ can be written as a union of sets $a + K_j$ with $a \in \Z[\i]$. 
\begin{align*}
	S(K_1) &= \bigg( 2 + K_3 \cup K_8 \bigg) \cup \bigcup_{\substack{n \in \Z \\ n \ge 3}} (n+K) \\
	S(K_2) &= \bigg( 2\!+\!\i + K_3 \cup K_8 \bigg) \cup \left( 1\!+\!2\i + \bigcup_{j=1}^{4} K_j \right) \\*&\qquad \cup \left( 2\!+\!2\i + \bigcup_{j=1}^{8} K_j \right) \cup \bigcup_{\substack{n+m\ssi \in \Z[\ssi] \\ \min\{m,n\} \ge 3}} (n\!+\!m\i + K) \\[-0.2em]
	S(K_3) &= \left( 2\i + \bigcup_{j=1}^4 K_j \right) \cup \bigg( 1\!+\!2\i + K_5 \cup K_6 \bigg) \\*&\qquad \cup \bigcup_{\substack{n\in\N\\ n\ge3}} \left( n\i + \bigcup_{j\ne 5,6,10} K_j \right) \cup \bigcup_{\substack{n\in\N\\ n\ge3}} \left( 1 + n\i + \bigcup_{j = 5,6,10} K_j \right) \\
	S(K_4) &= \bigg( 1\!+\!\i + K_3 \bigg) \cup \left( 2\!+\!\i + \bigcup_{j=1}^7 K_j \setminus K_3 \right) \\*&\qquad \cup \Big( 1\!+\!2\i + K_7 \cup K_8 \Big) \cup \Big( 2\!+\!2\i + K_9 \cup K_{10} \Big) \\
	S(K_5) &= \big( 1 + K_3 \big) \cup \big( 2 + K_4 \cup K_5 \big) \cup \big( 1\!+\!\i + K_8 \big) \cup \big( 2\!+\!\i + K_9 \cup K_{10} \big) \\
	S(K_6) &= 2 + K_1 \cup K_2 \cup K_6 \cup K_7
\end{align*}
On the right of \Cref{fig SH double}, the sets $S(K_1)$ through $S(K_6)$ are red, orange, yellow, teal, cyan, and light green, respectively.
By symmetry, expressions for $S(K_i), 7 \le i \le 10$, will be similar.
\end{proof}

From computer approximations, the sets $L_i$ for the shifted Hurwitz algorithm appear to be fractal. Unlike for standard Hurwitz (that is, nearest integer), some shifted Hurwitz $L_i$ are not bounded; this may be because the boundary of the fundamental set of the shifted Hurwitz algorithm contains many points with norm $1$ while the standard Hurwitz fundamental set is contained in $\overline{B(0,r)}$ with $r=\tfrac1{\sqrt2} < 1$.

\FloatBarrier

\end{document}